\documentclass{amsart}
\usepackage[utf8]{inputenc}
\usepackage{amssymb,amsmath,enumerate,amsthm}
\usepackage{colonequals}

\makeatletter
\@namedef{subjclassname@2020}{\textup{2020} Mathematics Subject Classification}
\makeatother

\usepackage[colorlinks]{hyperref}
\hypersetup{
 colorlinks=true,
 linkcolor=blue,
 filecolor=magenta, 
 urlcolor=cyan,
}

\PassOptionsToPackage{hyphens}{url}\usepackage{hyperref}

\usepackage[capitalise]{cleveref}
\crefformat{equation}{(#2#1#3)}
\crefrangeformat{equation}{(#3#1#4--#5#2#6)}
\crefformat{enumi}{(#2#1#3)}
\crefrangeformat{enumi}{(#3#1#4--#5#2#6)}

\usepackage[pagewise]{lineno}
\let\oldequation\equation
\let\oldendequation\endequation
\renewenvironment{equation}{\linenomathNonumbers\oldequation}{\oldendequation\endlinenomath}
\expandafter\let\expandafter\oldequationstar\csname equation*\endcsname
\expandafter\let\expandafter\oldendequationstar\csname endequation*\endcsname
\renewenvironment{equation*}{\linenomathNonumbers\oldequationstar}{\oldendequationstar\endlinenomath}
\let\oldalign\align
\let\oldendalign\endalign
\renewenvironment{align}{\linenomathNonumbers\oldalign}{\oldendalign\endlinenomath}
\expandafter\let\expandafter\oldalignstar\csname align*\endcsname
\expandafter\let\expandafter\oldendalignstar\csname endalign*\endcsname
\renewenvironment{align*}{\linenomathNonumbers\oldalignstar}{\oldendalignstar\endlinenomath}

\usepackage{tikz-cd}

\makeatletter

\newcommand{\xra}{\xrightarrow}

\newcommand{\dcat}[2][]{\cat{D}_{#1}(#2)}

\newcommand{\kinj}[1]{\cat{K}(\cat{Inj}\,#1)}
\newcommand{\perf}[1]{\mathsf{Perf}\,#1}

\newcommand{\coh}{\operatorname{coh}}
\newcommand{\dbcat}[1]{\cat{D}^{\mathsf{b}}(#1)}

\newcommand{\mcS}{\mathcal{S}}

\newcommand{\mco}{\mathcal{O}}

\newcommand{\fm}{\mathfrak{m}}

\newcommand{\fp}{\mathfrak{p}}

\newcommand{\sfc}{\mathsf{c}}

\newcommand{\bfi}{\boldsymbol{i}}
\newcommand{\bfq}{\boldsymbol{q}}

\newcommand{\vp}{{\varphi}}

\newcommand{\ann}{\operatorname{ann}}
\newcommand{\cat}[1]{{\mathsf{#1}}}
\newcommand{\codepth}{\operatorname{codepth}}
\newcommand{\cx}{\operatorname{cx}}
\newcommand{\depth}{\operatorname{depth}}
\newcommand{\edim}{\operatorname{edim}}
\newcommand{\Ext}{\operatorname{Ext}}

\newcommand{\flevel}{F\!\operatorname{-level}}
\newcommand{\level}{\mathsf{level}}
\newcommand{\hh}{\operatorname{H}}
\newcommand{\SheafHom}{\mathcal{H}\kern -1pt o\kern -1 pt m\kern -1pt}
\newcommand{\Hom}{\operatorname{Hom}}

\newcommand{\Ker}{\operatorname{Ker}}
\newcommand{\Loc}{\operatorname{Loc}}
\newcommand{\lotimes}{\otimes^{\operatorname{L}}}

\newcommand{\rmod}{\operatorname{mod}}
\newcommand{\one}{\mathbf{1}}

\newcommand{\pdim}{\operatorname{proj\,dim}}

\newcommand{\rank}{\operatorname{rank}}
\newcommand{\res}[2]{{#1}(#2)} 

\newcommand{\shift}{{\sf\Sigma}}
\newcommand{\spec}{\operatorname{Spec}}
 
\newcommand{\supp}{\operatorname{supp}}
\newcommand{\thick}{\operatorname{thick}}

\newcommand{\rmV}{\operatorname{V}}

\makeatother

\newcounter{intro}
\newtheorem{introthm}[intro]{Theorem}

\newtheorem{introcor}[intro]{Corollary}

\newtheorem{theorem}[subsection]{Theorem}
\newtheorem{proposition}[subsection]{Proposition}
\newtheorem{lemma}[subsection]{Lemma}
\newtheorem{corollary}[subsection]{Corollary}

\theoremstyle{definition}
\newtheorem{example}[subsection]{Example}
\newtheorem{definition}[subsection]{Definition}
\newtheorem{chunk}[subsection]{}

\theoremstyle{remark}

\newtheorem*{ack}{Acknowledgements}
\newtheorem{question}[subsection]{Question}

\numberwithin{equation}{subsection}

\title[Frobenius generators]{High Frobenius pushforwards generate the\\ bounded derived category}

\author[M.R.~Ballard]{Matthew~R.~Ballard}
\address{M.R.~Ballard, 
Department of Mathematics,
University of South Carolina, 
Columbia, SC 29208,
U.S.A.}
\email{mrb@mattrobball.com}

\author[S.B.~Iyengar]{Srikanth~B.~Iyengar}
\address{S.B.~Iyengar,
Department of Mathematics,
University of Utah,
Salt Lake City, UT 84112,
U.S.A.}
\email{srikanth.b.iyengar@utah.edu}

\author[P.~Lank]{Pat Lank}
\address{P.~Lank,
Dipartimento di Matematica ``F. Enriques", Universit\`{a} degli Studi
di Milano, Via Cesare
Saldini 50, 20133 Milano, Italy}
\email{plankmathematics@gmail.com}

\author[A.~Mukhopadhyay]{Alapan Mukhopadhyay}
\address{A.~Mukhopadhyay,
Institute of Mathematics, CAG, EPFL SB MATH MA A2 383 (Bâtiment MA), Station 8,
CH-1015 Lausanne, Switzerland}
\email{alapan.mathematics@gmail.com}

\author[J.~Pollitz]{Josh Pollitz}
\address{J.~Pollitz,
Mathematics Department, 
Syracuse University, 
Syracuse, NY 13244 U.S.A.}
\email{jhpollit@syr.edu}

\date{\today}



\begin{document}

\begin{abstract}
This work concerns generators for the bounded derived category of coherent sheaves over a noetherian scheme $X$ of prime characteristic. The main result is that when the Frobenius map on $X$ is finite, for any compact generator $G$ of $\mathsf{D}(X)$ the Frobenius pushforward $F ^e_*G$ generates the bounded derived category whenever $p^e$ is larger than the codepth of $X$, an invariant that is a measure of the singularity of $X$. The conclusion holds for all positive integers $e$ when $X$ is locally complete intersection. The question of when one can take $G=\mathcal{O}_X$ is also investigated. For smooth projective complete intersections it reduces to a question of generation of the Kuznetsov component.
\end{abstract}

\maketitle
\setcounter{tocdepth}{1}

\section*{Introduction}
\label{se:intro}

This work concerns the existence of (strong) generators in the bounded derived category of a noetherian scheme and, in particular, a commutative noetherian ring. The notion of a strong generator for an (essentially small) triangulated category $\cat T$ was introduced by Bondal and Van den Bergh~\cite{Bondal/VanDenBergh:2003}. Roughly speaking, an object $G$ in $\cat T$ is a generator if each object $A$ can be built from $G$ using the operations in $\cat T$: finite direct sums, summands, and mapping cones. When there is an an upper bound, independent of $A$, on the number of mapping cones required, $G$ is said to be a strong generator; see \cref{ch:thick}. 

Bondal and Van den Bergh~\cite{Bondal/VanDenBergh:2003} proved the existence of generators of derived categories of perfect complexes on quasi-separated and quasi-compact schemes. These are then generators for the bounded derived category of coherent sheaves when the scheme is smooth over a field. Rouquier~\cite{Rouquier:2008} subsequently showed that the bounded 
derived category can have strong generators even in the singular 
case. Further work of Keller and Van den Bergh~\cite{Keller/VanDenBergh:2009}, Lunts~\cite{Lunts:2010}, Iyengar and Takahashi~\cite{Iyengar/Takahashi:2016},  Neeman~\cite{Neeman:2021}, and Aoki~\cite{Aoki:2021} focused on the 
existence of strong generators in growing generality. 

While this settles the question of the existence of a generator 
for the bounded derived category of coherent sheaves quite broadly, there are few general results that identify them. Our motivation here is to identify, as canonically and explicitly as possible, generators. These 
will then immediately be strong whenever \cite{Aoki:2021} applies. We now
survey the current literature. 

We focus first on the affine situation: $X\colonequals \spec R$, where $R$ is a commutative noetherian ring. In this case, $\dbcat{\coh X}$ is equivalent to $\dbcat{\rmod R}$, the bounded derived category of finitely generated $R$-modules. When $R$ is regular and of finite Krull dimension, $R$ itself is a strong generator for $\dbcat{\rmod R}$. In fact, any $R$-complex in $\dbcat{\rmod R}$ having full support is a strong generator. The latter is a consequence of a theorem of Hopkins~\cite{Hopkins:1987}, and Neeman~\cite{Neeman:1992b}; see \cref{ch:perfect}. One can also identify strong generators when $R$ is a complete intersection ring with isolated singularities using a theorem of Stevenson~\cite{Stevenson:2014a}. When $R$ is an artinian ring, $R/J$, where $J$ is the Jacobson radical of $R$, is a strong generator for $\dbcat{\rmod R}$. 

The situation is even more complicated in the global case. Perhaps the most comprehensive result is due to Orlov~\cite{Orlov:2009} who proved that when $X$ is quasi-projective, with very ample line bundle $L$, the direct sum $\mco_X \oplus L\oplus \cdots \oplus L^{\otimes d}$, where $d=\dim X$, is a generator for $\perf{X}$, the subcategory of perfect complexes. This gives strong generators for $\dbcat{\coh X}$ when $X$ is regular. Outside this case, there are only a few results that identify explicit generators for the bounded derived category; see, for instance, the work of \cite{Ohkawa/Uehara:2013}.

In light of these remarks, it is surprising that for any $F$-finite scheme of prime characteristic one can identify strong generators for the bounded derived category, at least in terms of generators for $\perf X$. This is the content of the result below. The existence of a generator $E$, as below, for $\perf X$ is a result of Bondal and Van den Bergh~\cite[Theorem~3.1.1]{Bondal/VanDenBergh:2003}.  We emphasize that $E$ need not be a strong generator for $\perf X$. Indeed, the latter has no strong generators when $X$ is singular; see \cite{Rouquier:2008,Oppermann/Stovicek:2012}.
As usual $F^e_*G$ denotes the pushforward of a complex $G$ in $\dbcat{\coh X}$ along $F^e$, the $e$-fold composition of the Frobenius map on $X$. 

\begin{introthm}
\label{ith:general-schemes}
Let $X$ be a noetherian $F$-finite scheme of prime characteristic $p$, and $E$ a generator for $\perf X$. For any $G$ in $\dbcat{\coh X}$ with $\supp_X G=X$, the complex $F^e_*(E\lotimes_X G)$ is a generator for $\dbcat{\coh X}$ for any natural number $e> \log_p(\codepth X)$; it is a strong generator when $X$ is separated.
\end{introthm}

This result is contained in \cref{co:strong-generation}. The invariant $\codepth X$ is defined for any noetherian scheme $X$, in terms of the local rings at $x\in X$; see \cref{se:generation}. When $X\colonequals \spec R$, the codepth of $X$ is bounded above by the number of generators of the $R$-module $F _*R$; see \cref{le:codepth}, which also provides a better bound. In particular, the codepth is finite for any $F$-finite noetherian scheme. The codepth of $X$ is equal to $0$ if and only if $X$ is regular, and so it can be viewed as a numerical measure of the singularity of $X$.

Two special cases are worth noting. When $X$ is affine, the structure sheaf $\mco_X$ of $X$ generates $\perf X$, so the preceding result specializes to:

\begin{introcor}
 \label{icor:affine}
When $R$ is an $F$-finite commutative noetherian ring of prime characteristic $p$, the $R$-module $F^e_*R$ is a strong generator for $\dbcat{\rmod R}$ for any natural number $e>\log_p(\codepth R)$.
\end{introcor}

In fact, one can take the Frobenius pushforward of any $R$-complex with coherent cohomology and full support; see \cref{co:strong-generation-affine}. Here is another special case of \cref{ith:general-schemes}; it reappears as \cref{co:quasi-proj}.

\begin{introcor}
\label{icor:quasi-projective}
When $X$ is a quasi-projective scheme over an $F$-finite field of prime characteristic $p$, and $L$ a very ample line bundle, $F^e_* \mco_X \oplus F^e_* L \oplus \cdots \oplus F^e_* L^{\otimes \dim X}$ is a strong generator of $\dbcat{\coh X}$, for any natural number $e> \log_p(\codepth X)$.
\end{introcor}

It is not hard to deduce from these results that if for some $G$ in $\dbcat{\coh X}$ with full support and integer $n\ge 1$, the complex $F^n_*G$ is perfect, then $X$ is regular; this recovers \cite[Theorem~1.1]{Avramov/Hochster/Iyengar/Yao:2012}. Thus one can view \cref{ith:general-schemes} as a structural refinement of Kunz's theorem~\cite{Kunz:1969} that is the case $G=\mco_X$. One can also deduce other results of this ilk, characterizing the singularity type of $X$ in terms of homological properties of $F$; such as those from \cite{Avramov/Hochster/Iyengar/Yao:2012,Herzog:1974,Rodicio:1988,Takahashi/Yoshino:2004, Koh/Lee:1998} from our results.

There are two key ingredients that go into the proof of \cref{ith:general-schemes}. One is a local-to-global principle that says, roughly, that a complex $G$ is a generator for $\dbcat{\coh X}$ if the stalk complex $G_x$ is a generator for $\dbcat{\rmod \mco_x}$ for each $x\in X$; see \cref{co:local-global}. A subtle point here is that generation involves the structure of $\dbcat{\coh X}$ as module over $\perf X$ viewed as a tensor-triangulated category. We deduce this local-to-global principle from a more general statement, involving the ind-completions of the categories. This builds on work of Stevenson~\cite{Stevenson:2013a}, and is explained in \cref{se:local-global}.

Once we are down to the level of stalks, the main result concerns a nilpotence property of the Frobenius endomorphism on a local ring of prime characteristic. It implies that, under appropriate conditions involving support, high Frobenius pushforwards of any complex generates the finite length complexes in $\dbcat{\rmod \mco_x}$. This is the content of \cref{se:nilpotence}.

\cref{ith:general-schemes} prompts a number of questions. One is what can be said, vis a vis generation, when $F $ is not necessarily finite; see \cref{th:kinj-generation} for what we have to offer in this regard. Another question, prompted by Kunz's theorem, is whether already the first Frobenius pushforward of some object generates the bounded derived category. Here is the most decisive result we prove concerning this question. 

\begin{introthm}
 \label{ith:completeintersection}
When $X$ is $F$-finite and locally complete intersection, for any generator $G$ of $\perf X$, the complex $F_*G$ is a strong generator for $\dbcat{\coh X}$. In particular, if $X$ is affine, $F_*\mco_X$ strongly generates $\dbcat{\coh X}$.
\end{introthm}

See \cref{th:completeintersection}, and also \cref{th:completeintersection-2} for a more general statement. We also prove that Veronese subrings of $k[x,y]$ have this property. The results for affine varieties raises the question: Does $F_{*}R$ generate $\dbcat{\rmod R}$ for any commutative noetherian ring $R$? We do not know of any counterexamples. 

In light of these results it is natural to ask for the smallest number $e$ such that $F^e_*\mco_X$ strongly generates $\dbcat{\coh X}$. This is interesting even when $X=\spec R$, for examples suggest that the upper bound for $e$ given by \cref{ith:general-schemes} is far from optimal. Here are some results in this direction; see \cref{se:F-thickness,se:firstpushforward}.
\begin{itemize}
\item
$e\le \lceil\log_p(\mathrm{Loewy\,length}\, R)\rceil$ when $R$ is an artinian local ring; 
\item
$e\le \lceil \log_p(n+1)\rceil$ when $X=\mathbb{P}^n$;
\item
$e\le \lceil \log_p3\rceil$ with $X$ the blowup of $\mathbb{P}^2$ at $\le 4$ points in general position;
\item
$e\le \lceil\log_pn\rceil$ for some smooth quadrics $X$ of dimension $n$.
\end{itemize}
 There may be no such $e$ when, for instance, $X$ is an $F$-finite smooth curve of positive genus; see \cref{th:f_thick_curves}. This prompts a definition: $X$ is \emph{$F$-thick} if $F^e_*\mco_X$ itself generates $\dbcat{\coh X}$ for some $e\ge 1$. The class of $F$-thick schemes includes all affine schemes and overlaps significantly with varieties possessing a full exceptional collection. For smooth projective complete intersections $F$-thickness is tantamount to generation of the Kuznetsov component; see \cref{th:f_split_kuz_comp}.

Shifting the focus from the full derived category to specific objects, recent and not so recent developments suggest considering the number of steps required to build $\mco_X$ from $F^e_*\mco_X$, for some $e$. We are interested in this number partly because in the affine case it is one measure of a failure of the $F$-split property.

\begin{introcor}
 \label{icor:fLevelFSplit}
For any $F$-finite noetherian ring $R$ one has that
\[
\inf\{n\mid \text{$R$ is in $\thick^n(F^e_*R)$ for some $e\ge 1$}\}
\]
is finite; it equals one if and only if $R$ is $F$-split. It is bounded above by $p^c$ when $R$ is locally complete intersection and $c$ is the codimension of $R$.
\end{introcor}

The finiteness is immediate from \cref{icor:affine}, and the second part is essentially the definition of the $F$-split property. These results are proved in \cref{se:new_invariants}. Another reason for our interest in \cref{icor:fLevelFSplit}  is that it says, in the language of \cite{Dwyer/Greenlees/Iyengar:2006b}, that the $R$-module $F^e_*R$ is proxy-small for $e\gg 0$. Among other things, this has consequences for the derived endomorphism ring of $F^e_*R$; see \emph{op.\@ cit.\@} and also \cite{Dwyer/Greenlees/Iyengar:2006a}.

\begin{ack}
Iyengar thanks Greg Stevenson for conversations on \cref{se:local-global}. Mukhopadhyay thanks the University of Utah mathematics department for hospitality during formative stages of the work. The authors thank  Devlin Mallory and Eamon Quinlan-Gallego for comments on the material in  \cref{se:F-thickness} and \cref{se:new_invariants}, respectively, and Roman Bezrukavnikov for pointing us to \cite{Bezrukavnikov/Mirkovic/Rumynin:2008}. Finally, we are grateful to the anonymous referees for their helpful comments and suggestions.  
\end{ack}

\section{Generators for triangulated categories}
\label{se:local-global}

In this section we recall some basic notions and results concerning triangulated categories. The main examples of interest are various categories constructed from (quasi-)coherent sheaves on a noetherian scheme. We take Krause's book~\cite{Krause:2022} as our standard reference on triangulated categories; for constructions and results specific to schemes see Huybrecht's book~\cite{Huybrechts:2006} and the Stacks Project~\cite{StacksProject}.

\begin{chunk}
\label{ch:thick}
Let $\cat K$ be a triangulated category. A subcategory $\cat{S}$ of $\cat K$ is \emph{thick} if it is a full triangulated subcategory closed under retracts. Given an object $G$ (or a set of them) in $\cat K$, we write $\thick_{\cat K}(G)$ for the smallest thick subcategory, with respect to inclusion, containing $G$. The objects in $\thick_{\cat K}(G)$ are referred to as being \emph{finitely built from $G$}. This subcategory is also often denoted $\langle G \rangle$; see, for example, \cite{Bondal/VanDenBergh:2003}. One can construct $\thick_{\cat K}(G)$ inductively as follows: set $\thick_{\cat K}^1(G)$ to be the smallest full subcategory of $\cat K$ containing $G$ and closed under suspensions, finite sums and retracts. For $n> 1$, let $\thick_{\cat K}^n(G)$ be the smallest full subcategory of $\cat K$ containing all objects $A$ that fit into a triangle 
\[
A\to B\to C\to \shift A\,
\] 
with $B$ in $\thick_{\cat K}^{n-1}(G)$ and $C$ in $\thick_{\cat K}^1 (G)$, and which is closed under suspensions, finite sums and retracts. One has a filtration
\[
\thick_{\cat K}(G) =\bigcup_{n\geqslant 1}\thick_{\cat K}^n(G)\,.
\]
See \cite{Bondal/VanDenBergh:2003}, and also \cite{Avramov/Buchweitz/Iyengar/Miller:2010}, for details. An object $G$ is a \emph{(classical) generator} for $\cat K$ if $\thick_{\cat K}(G)=\cat K$; it is a \emph{strong generator} provided that $\thick_{\cat K}^d(G)=\cat K$ for some integer $d$. Clearly, if $\cat K$ admits a strong generator, any generator is a strong generator. 

Suppose $\cat K$ admits all coproducts. A triangulated subcategory of $\cat K$ is \emph{localizing} when it is closed under small coproducts. We write $\Loc_{\cat K}(G)$ for the smallest localizing subcategory containing $G$, and speak of objects in this category as being \emph{built} from $G$. An object $C$ in $\cat K$ is \emph{compact} if $\Hom_{\cat K}(C,-)$ commutes with coproducts in $\cat K$, and $\cat K$ is \emph{compactly generated} if there exists a set of compact objects $G$ such that $\Loc_{\cat K}(G)=\cat K$. We write ${\cat K}^{\mathrm c}$ for the compact objects in $\cat K$; this is a thick subcategory of $\cat K$. If $G$ is a set of compact objects in $\cat K$, then
\begin{equation}
\label{eq:neeman-cgt}
\Loc_{\cat K}(G) \cap {\cat K}^{\mathrm c}= \thick_{\cat K}(G)\,.
\end{equation}
This is proved by Neeman~\cite[Theorem~2.1]{Neeman:1992a}; see also ~\cite[Proposition~3.4.15]{Krause:2022}.
\end{chunk}

Many of the triangulated categories of interest in this manuscript come equipped with an action of another triangulated category, and our proofs exploit this additional structure. The relevant notions are recalled below. For details see \cite{Stevenson:2013a}.

\begin{chunk}
\label{ch:actions}
Let $(\cat T,\otimes,\one)$ be a tensor triangulated category that acts on $\cat K$ (on the left), in the sense of Stevenson~\cite[Section~3]{Stevenson:2013a}; see also \cite{Stevenson:2018a}. We write $\odot$ to denote this action. It is suggestive to think, and speak, of $\cat K$ as a $\cat T$-module, as in \cite{Stevenson:2013a,Stevenson:2014b}. Given an object $G$ (or, as before, a set of them) in $\cat K$, we denote $\thick^{\odot}_{\cat K}(G)$ the thick \emph{$\cat T$-submodule of $\cat K$ generated by $G$}, that is to say, the smallest thick subcategory of $\cat K$ that is closed under the action of $\cat T$.

When $\cat T$ and $\cat K$ admit all coproducts, we consider also $\Loc^{\odot}_{\cat K}(G)$, the \emph{localizing $\cat T$-submodule} of $\cat K$ generated by $G$. If  $\Loc_{\cat T}(U)=\cat T$ for a set of objects $U$, then
\begin{equation}
\label{eq:localizing-module}
\Loc^{\odot}_{\cat K} (G)=\Loc_{\cat K}(A\odot B\mid A\in U \text{ and } B\in G)\,. 
\end{equation}
In particular, if $\Loc_{\cat T}(\one)=\cat T$, then $\Loc^{\odot}_{\cat K}(G)=\Loc_{\cat K}(G)$; see \cite[Lemma~3.13]{Stevenson:2013a}.

Suppose that $\cat T$ is compactly generated. Then Brown representability yields that $\cat T$ has an internal function object, $\mathrm{hom}(-,-)$, adjoint to $-\otimes-$; we assume that this is exact in each variable. An object $D$ in $\cat T$ is \emph{rigid} if for each $E$ in $\cat T$ the natural map
\[
\mathrm{hom}(D,\one)\otimes E \longrightarrow \mathrm{hom}(D,E)
\]
is an isomorphism. The category $\cat T$ is \emph{rigidly compactly generated} if it is compactly generated, and the set of compact objects and rigid objects coincide; see \cite[Defintion~4.1]{Stevenson:2013a}. In this case, the unit $\one$ of $\cat T$ is compact. When a rigidly compactly generated category $\cat T$ acts on a compactly generated category $\cat K$, the action restricts to compact objects, that is to say, ${\cat T}^{\mathrm c}$ acts on ${\cat K}^{\mathrm c}$; see \cite[Lemma~4.6]{Stevenson:2013a}.

These observations yield also a module version of \cref{eq:neeman-cgt}, namely, when $\cat T$ is rigidly compactly generated and $\cat K$ is compactly generated, for any set $G$ of compact objects in $\cat K$, one has
\begin{equation}
\label{eq:neeman-tt}
\Loc^{\odot}_{\cat K}(G) \cap {\cat K}^{\mathrm c}= \thick^{\odot}_{\cat K}(G)\,, 
\end{equation}
where the category on the right is the $\cat T^{\mathrm c}$-submodule of ${\cat K}^{\mathrm c}$ generated by $G$. This follows from \eqref{eq:neeman-cgt}, given \eqref{eq:localizing-module}.
\end{chunk}

\begin{chunk}
\label{ch:BondalVanDenBergh}
Let $X$ be a noetherian scheme and $\dcat X$ the derived category of quasi-coherent sheaves, viewed as a triangulated category with suspension functor $\shift$. When $X=\spec R$ is affine, we identify $\dcat{X}$ with $\dcat{R}$, the derived category of $R$-modules. 

As a triangulated category $\dcat X$ admits arbitrary coproducts, and is compactly generated.
The compact objects in $\dcat X$ are the bounded complexes of vector bundles; that is to say, \emph{perfect} complexes. We write $\perf X$ for the full subcategory consisting of perfect complexes. Bondal and Van den Bergh~\cite[Theorem~3.1.1]{Bondal/VanDenBergh:2003}, see also \cite[Theorem~3.1.3]{Bondal/VanDenBergh:2003} or \cite[\href{https://stacks.math.columbia.edu/tag/09T4}{Tag 09T4}]{StacksProject}, proved that $\dcat X$ has a compact generator: a perfect complex $G$ such that 
\[
\dcat X=\Loc_{\dcat X}(G)\,.
\]
It follows that $\perf X=\thick_{\dcat X}(G)$; see \cref{eq:neeman-cgt}.

Our main interest is in $\dbcat{\coh X}$, the full subcategory of $\dcat X$ consisting of bounded complexes with coherent cohomology, and in finding strong generators for this triangulated category. Aoki~\cite{Aoki:2021} proved that these exist whenever $X$ is separated, quasi-excellent, and of finite Krull dimension.
\end{chunk}

\begin{chunk}
The derived tensor product $-\lotimes_X -$ endows $\dcat X$ with a structure of a rigidly compactly generated tensor triangulated category. Moreover $\perf X$ is a tensor triangulated subcategory of $\dcat X$ and the derived tensor product gives an action of $\perf X$ on $\dcat X$, in the sense of \cref{ch:actions}. One has
\[
\Loc^{\odot}_{\dcat X}(G) = \Loc^{\otimes}_{\dcat X}(G)
\]
for any $G$ in $\dcat X$, where the category on the left is the localizing $\perf X$-submodule of $\dcat X$ generated by $G$, and the category on the right is the localizing tensor ideal of $\dcat X$. Moreover, when $X$ is affine, the unit of the tensor product- $\mco_X$, generates $\perf{X}$ as a thick subcategory, so 
\[
\Loc^{\odot}_{\dcat X}(G)=\Loc_{\dcat X}(G)\,.
\]
These observations will be used implicitly in the sequel.

The triangulated subcategory $\dbcat{\coh X}$ of $\dcat X$ is not closed under the tensor product, unless $X$ is regular; see \cref{ch:regular}. However, the tensor product on $\dcat X$ restricts to an action of $\perf X$ on $\dbcat{\coh X}$, so we are in the context of \cref{ch:actions}.
\end{chunk}

The next paragraph is a recap on some results involving support for objects in $\dcat X$; for details see \cite[Appendix A]{Iyengar/Lipman/Neeman:2015}, or \cite{StacksProject}. This builds on the theory of support for complexes over commutative noetherian rings developed by Foxby~\cite{Foxby:1979}.

\begin{chunk}
\label{ch:perfect}
Let $X$ be a noetherian scheme as before, and fix $E$ in $\dcat X$. Given $x\in X$ we write $E_x$ for the stalk of $E$ at $x$, viewed as an object in the derived category of the local ring $\mco_{x}$. Let $k(x)$ denote the residue field of $\mco_{x}$. We identify $k(x)$ with the coherent sheaf on $\spec \mco_{x}$ it defines, as well as with the pushforward of this coherent sheaf along the localizing immersion $\spec\mco_{x}\to X$. The \emph{support} of the complex $E$ is the subset of $X$ prescribed by
\[
\supp_X E \colonequals \{x\in X\mid \hh(E\lotimes_X k(x))\ne 0\}\,.
\]
This is sometimes referred to as the \emph{small support} of $E$. When $E$ is in $\dbcat{\coh X}$ its support coincides with the usual one:
\[
\supp_X E = \{x\in X\mid \hh(E_x)\ne 0\}\,.
\]
In particular, in this case it is a closed subset of $X$.

When $X=\spec R$, with $R$ a commutative noetherian ring and $E$ is the sheafification $\widetilde{M}$ of an $R$-complex $M$, this is the notion of support from \cite{Foxby:1979}, denoted $\supp_RM$. When moreover $M$ is in $\dbcat{\rmod R}$ one has
\[
\supp_X E = \supp_RM = V(\ann_R\hh(M))\,.
\]

Our interest in support stems from the following result, proved by Neeman~\cite{Neeman:1992b} in the affine case, and extended to schemes by Alonso Tarr\'io, Jerem\'ias L\'opez, and Souto Solario~\cite{Tarrio/Lopez/Solario:2004}: Given objects $E,G$ in $\dcat X$, one has
\begin{equation}
\label{eq:neeman}
E\in \Loc^{\otimes}_{\dcat X}(G)\Longleftrightarrow \supp_XE\subseteq \supp_XG\,.
\end{equation}

Using this result and \cref{eq:neeman-cgt} one deduces that when $E,G$ are in $\perf X$ one has
\begin{equation}
\label{eq:hopkins}
E\in \thick^{\otimes}_{\perf X}(G)\Longleftrightarrow \supp_XE\subseteq \supp_XG\,.
\end{equation}
This result was proved by Thomason~\cite{Thomason:1997}. The affine case is a result of Hopkins \cite{Hopkins:1987} and Neeman \cite{Neeman:1992b}. From \eqref{eq:hopkins} it follows, for example, that when $E$ in $\perf X$ has full support then it generates $\perf X$ as a module over itself. 
\end{chunk}

\begin{chunk}
Given $x\in X$ let $K$ be the Koszul complex on a finite generating set for the maximal ideal of the local ring $\mco_{x}$, and for each $E$ in $\dcat X$ set
\[
\res{E}x \colonequals E_x\otimes_{\mco_x}K\,,
\]
viewed as an element in $\dcat{\mco_{x}}$. This notation is ambiguous, for it does not reflect the choice of a generating set for the maximal ideal. However, if $K'$ is the Koszul complex on a different generating set, then 
$\thick(K)=\thick(K')$ as subcategories of $\dcat{\mco_x}$; this follows from \eqref{eq:hopkins}, but can also be checked directly, as is done in \cite[Lemma~6.0.9]{Hovey/Palmieri/Strickland:1997}. Consequently $\thick(\res{E}x)$ is well-defined, and this is what is relevant in the statement below, and also later on. Observe that, since $K$ is a perfect $\mco_{x}$-complex, when $E$ is in $\dbcat{\coh X} $, the complex $\res{E}x$ is in $\dbcat{\rmod \mco_{x}}$. 

\begin{theorem}
\label{th:local-global}
Let $X$ be a noetherian scheme. Fix objects $E,G$ in $\dbcat{\coh X}$. If $\res{E}x$ is in $\thick_{\dcat{\mco_x}}( G_x)$ for each $x\in X$, then $E$ is in $\thick^{\odot}_{\dcat X}(G)$.
\end{theorem}

A proof is given further below; a input in it is the work of Stevenson~\cite{Stevenson:2014a}.  Observe that the converse also holds: if $E$ is in the $\perf X$-submodule generated by $G$, then since both $(-)_x$ and $-\otimes_{\mco_x}K$ are exact, $\res{E}x$ is generated by $\res{G}x$, and hence also by $G_x$, as $\mco_x$ finitely builds $K$.

\begin{chunk}
\label{ch:kinj-X}
Let $\kinj{ X}$ denote the homotopy category of quasi-coherent injective sheaves on $X$. This is a compactly generated triangulated category and restricting the natural localization functor 
\[
\bfq\colon \kinj{X}\longrightarrow \dcat X
\]
to the subcategory of compact objects induces an equivalence
\begin{equation}
\label{eq:kinjc}
\kinj{ X}^{\mathsf c}\xrightarrow{\ \sim\ }\dbcat{\coh X}\,.
\end{equation}
See \cite[Theorem~1.1]{Krause:2005}, where these results are proved when $X$ is separated, and \cite[Appendix~B]{Coupek/Stovicek:2020} for the general case.

The tensor triangulated category $\dcat{X}$ is rigidly compactly generated and acts on $\kinj{X}$, in the sense of \cref{ch:actions}. One can view this action as the ind-completion of the action of $\cat{Perf}(X)$ on $\dbcat{\coh X}$. Here is a concrete description of this action, following \cite[Section~3]{Stevenson:2014a}.

Given a flat $\mco_X$-module $F$ and an injective $\mco_X$-module $I$, the $\mco_X$-module $F\otimes_XI$ is injective since $X$ is noetherian. Thus, the tensor product induces an action of the homotopy category of complexes of flat modules on $\kinj X$. When $F$ is an acyclic complex of flat modules with flat syzygies (also known as a pure acyclic complex), and $I$ is a complex of injectives, $F\otimes_XI$ is contractible. Thus the said action factors through the Verdier quotient of the homotopy category of flats by the subcategory of pure acyclic complexes. This category, introduced by Murfet~\cite{Murfet:2007, Neeman:2010}, is usually denoted $\cat{N}(X)$ and called the Neeman category of $X$. When $X$ is affine this is equivalent to the homotopy category of projective modules. A complex $F$ of $\mco_X$-modules is \emph{$K$-flat} if it consists of flat $\mco_X$-modules and $F\otimes_X-$ preserves quasi-isomorphisms. Taking $K$-flat resolutions gives a fully faithful embedding of $\dcat X$ into $\cat N(X)$, compatible with tensor products and coproducts. Via this embedding one gets an action of $\dcat X$ on $\kinj X$.

In what follows, we write $\Loc^{\odot}_{\cat K(X)}(G)$, with $\cat K(X)$ in the subscript rather than $\kinj X$, for the localizing submodule of $\kinj X$ generated by an object $G$. 

Fix $x\in X$ and let $\varGamma_x$ denote the exact functor from $\dcat X$ to $\dcat{\mco_x}$ defined by the assignment
\[
E\longrightarrow \mathbf{R}\varGamma_{\overline{\{x\}}}(E_x)\,.
\]
 See \cite[Definition~5.3]{Stevenson:2013a}. The local-to-global principle~\cite[Theorem~6.9]{Stevenson:2013a} yields that given $E,G$ in $\kinj{X}$, viewed with the $\perf X$-action, one has
\begin{equation}
\label{eq:local-global}
E\in \Loc^{\odot}_{\cat K(X)}(G) \Longleftrightarrow
\varGamma_x E\in \Loc_{\cat K(\mco_x)}(\varGamma_xG) \quad \text{for each $x\in X$}.
\end{equation}
The condition on the right is equivalent to $\varGamma_xE$ being in the localizing subcategory generated by $G_x$, for the functor $\varGamma_x(-)$ commutes with coproducts. For another version of this local-to-global principle see \cite[Lemma~3.2]{Benson/Iyengar/Krause:2011a}.
\end{chunk}

\begin{proof}[Proof of ~\cref{th:local-global}]
The basic idea is to apply the local-to-global principle. We use the equivalence \cref{eq:kinjc} and identify $E$ and $G$ with their images in $\kinj{X}$. The desired conclusion is that $E$ is in the localizing $\dcat X$-submodule of $\kinj{X}$ generated by $G$; the statement about thick submodules follows by \cref{eq:neeman-tt}. For any $x$ in $X$ one has that
\[
\Loc_{\cat K(\mco_x)}(\res {E}x) = \Loc_{\cat K(\mco_x)}(\varGamma_xE)\,.
\]
This holds because $\res {\mco}x$, the Koszul complex of the local ring $\mco_x$, and $\varGamma_x\mco_x$ build each other in $\dcat{\mco_x}$; indeed both objects in question are supported on $\{x\}$, so \cref{eq:neeman} applies. Thanks to the tensor action of $\dcat{\mco_x}$ on $\kinj{\mco_x}$ the equality above holds. One can also give a direct, elementary proof; see, for example, \cite[Proposition~2.11]{Benson/Iyengar/Krause:2011a}. Given this observation, the hypothesis implies
\[
\varGamma_xE\in \Loc_{\cat K(\mco_x)}(G_x) \quad \text{for each }x\in X.
\]
At this point we can invoke the local-to-global principle \eqref{eq:local-global} to deduce that $E$ is in the $\dcat X$-submodule of $\kinj X$ generated by $G$, as desired.
\end{proof}
\end{chunk}

\begin{chunk}
Fix $x\in X$, set $R\colonequals \mco_x$ and $k\colonequals k(x)$. We write $\varGamma_x\dbcat{\rmod R}$ for the subcategory of $\dbcat{\rmod R}$ consisting of $R$-complexes $M$ supported at the closed point $x$ of $\spec R$; equivalently, the $R$-module $\hh(M)$ has finite length. This subcategory contains $\res{E}x$ for all $E$ in $\dbcat{\coh X}$. The residue field $k(x)$ of $R$ generates $\varGamma_x\dbcat{\rmod R}$; see, for example, \cite{Dwyer/Greenlees/Iyengar:2006b}. However $k(x)$ is a strong generator for this category if and only if $R$ is artinian; see, for instance, \cite{Rouquier:2008}. 
\end{chunk}

Given these observations the result below is a consequence of \cref{th:local-global}.

\begin{corollary}
\label{co:local-global}
Let $X$ be a noetherian scheme and fix $G$ in $\dbcat{\coh X}$. If for each $x\in X$, the complex $G_x$ finitely builds $k(x)$, 
then $G$ generates $\dbcat{\coh X}$ as a $\perf X$-module. \qed
\end{corollary}

One reason for our interest in finding generators of derived categories is that they are test objects for finiteness of various homological invariants. A sample result along these lines is provided in \cref{le:sample}; see also \cite{Dwyer/Greenlees/Iyengar:2006b}. To illustrate this point, we recall a classical characterization of regularity.

\begin{chunk}
\label{ch:regular}
Let $X$ be a noetherian scheme. Clearly $\perf X\subseteq \dbcat{\coh X}$. Equality holds precisely when $X$ is regular, that is to say, the local rings $\mco_{x}$ are regular for $x$ in $X$; equivalently, at each closed point $x$ in $X$. This characterization of regularity is  due to Auslander, Buchsbaum, and Serre, given the following observation: A complex $E$ in $\dbcat{\coh X}$ is perfect if and only if $E_x$ is perfect in $\dbcat{\rmod \mco_x}$ for each $x\in X$; equivalently, at each closed point $x\in X$. In the affine case, which implies the global version, this is due to Bass and Murthy \cite{Bass/Murthy:1967}; see also \cite[Theorem~4.1]{Avramov/Iyengar/Lipman:2010}.
 
Here is a related result: Any compact generator for $\dcat X$ generates $\perf X$, as a thick subcategory. However $\perf X$ has a strong generator if and only if $X$ is regular and of finite Krull dimension; see \cite[Proposition~7.25]{Rouquier:2008}.
\end{chunk}

\begin{lemma}
\label{le:sample}
Let $X$ be a noetherian scheme and $G$ a generator for $\dbcat{\coh X}$. An object $E$ in $\dbcat{\coh X}$ is perfect if and only if $\Ext_X^i(E,G)=0$ for $i\gg 0$. In particular, if $\Ext_X^i(G,G)=0$ for all $i\gg 0$, then $X$ is regular. 
\end{lemma}

\begin{proof}
The second assertion follows immediately from the first, where the only if direction is clear. Suppose $\Ext_X^i(E,G)=0$ for $i\gg 0$. Since $G$ generates the bounded derived category, for each closed point $x\in X$, the coherent sheaf $k(x)$ is finitely built from $G$ in $\dcat X$, hence the hypothesis yields
\[
\Ext^i_{\mco_x}(E_x,k(x)) \cong \Ext^i_{X}(E,k(x)) =0 \quad \text{for $i\gg 0$.}
\]
It follows that the $\mco_x$-module $E_x$ is perfect. Thus $E$ is perfect; see \cref{ch:regular}.
\end{proof}

\section{Local rings}\label{se:nilpotence}

In this section we establish a nilpotence-type result concerning the Frobenius endomorphism on a local ring. This is one of the key inputs into our arguments, in the next section concerning generators for the bounded derived category of a scheme over a field of positive characteristic.

Let $R$ be a noetherian ring of prime characteristic $p$; that is to say, $p$ is a prime number and $R$ contains the field $\mathbb{F}_p$ as a subring. Let
\[
F\colon R\to R\quad\text{given by}\quad r\mapsto r^{p}
\]
be its Frobenius endomorphism. It is well-known that the action induced by the Frobenius on various homology modules is often trivial; this springs from the fact that the Frobenius endomorphism on any simplicial commutative ring induces the trivial map in homology; see \cite{Pirashvili:2007} and also \cite[Section~11]{Bhatt/Scholze:2017a}. The gist of the result below is that, on noetherian local rings, the Frobenius endomorphism is even essentially nilpotent; see also \cref{pr:nilpotence}.

\begin{theorem}
\label{th:nilpotence}
Let $R$ be a noetherian local ring of prime characteristic $p$. For any $R$-complex $M$ and natural number $e> \log_p(\codepth R)$ there is an isomorphism 
\[
F ^e_*(K^M) \simeq \hh(F ^e_*(K^M)) \quad\text{in $\dcat R$.}
\]
In particular, the residue field $k$ is in $\thick_{\dcat R}(F ^e_*(K^M))$ whenever $\hh(K^M)\ne 0$.
\end{theorem}

In the statement, $K^M\colonequals K\otimes_RM$ where $K$ is the Koszul complex on minimal generating set for the maximal ideal of $R$, and $F ^e_*$ is the restriction of scalars functor along $F ^e$. The \emph{codepth} of $R$ is the non-negative integer
\[
\codepth R\colonequals \edim R -\depth R=\sup\{i\mid \hh_i(K^R)\ne 0\}\,.
\]
The equality on the right is by the depth sensitivity of the Koszul complex. Let $Q\twoheadrightarrow R$ be a Cohen presentation, meaning that $Q$ is a regular local ring and the map is surjective. One has
\begin{equation}
\label{eq:codepth-pdim}
\codepth R \le \pdim_Q R\,,
\end{equation}
and equality holds if (and only if) $\edim Q=\edim R$, that is to say, when $Q\twoheadrightarrow R$ is a minimal Cohen presentation.
This holds because of the formula of Auslander and Buchsbaum~\cite[Theorem~1.3.3]{Bruns/Herzog:1998}.

Since complete local rings admit Cohen presentations, this remark can be applied after completing $R$ at its maximal ideal, for the codepth of $R$ coincides with that of its completion; indeed, both the embedding dimension and the depth remain unchanged in this process.

\begin{chunk}
\label{ch:AIM}
Compare the statement of \cref{th:nilpotence} with \cite[Theorem~6.2.2]{Avramov/Iyengar/Miller:2006} that gives the same conclusion but where the lower bound for $e$ is the spread of $R$. This number can be computed in terms of the graded Betti-numbers of $\mathrm{gr}_{\fm}(R)$, the associated graded ring of $R$ at its maximal ideal $\fm$, over the symmetric algebra on $\fm/\fm^2$, and is related to the regularity of $\mathrm{gr}_{\fm}(R)$. The spread is harder to control; for instance, when $R$ is a (not necessarily local) noetherian ring essentially of finite type, we do not know whether there is a global bound on the spread of $R_\fp$ as $\fp$ varies over the prime ideals in $R$. Such a bound is clear for the codepth and this fact is important in the sequel.

A key difference between the techniques used in \cite[Theorem~6.2.2]{Avramov/Iyengar/Miller:2006} and our technique in \cref{th:nilpotence} is that we exploit the special properties of the Frobenius map by working in category of simplicial algebras and simplicial modules over them. This yields our codepth bound in \cref{th:nilpotence}. The pertinent constructions and results concerning simplicial algebras are recalled below; see \cite{Quillen:1967,Schwede/Shipley:2003} for proofs.
\end{chunk}

\begin{chunk}
\label{ch:dold-kan}
Given a simplicial ring $\mathcal A$ we write $\mathrm{Simp}(\mathcal A)$ for the category of simplicial $\mathcal A$-modules, with the usual model structure, and $\mathrm{Ho}(\mathcal A)$ for the corresponding homotopy category.

Viewing a commutative ring $R$ as a simplicial ring in the standard way, by the Dold-Kan theorem one gets an equivalence of categories
\[
 \begin{tikzcd}
 \mathrm{Simp}(R) \ar[r, shift right=1ex, "N" swap, "\sim"] 
 & \mathrm{Ch}_{\geqslant 0}(R)\,, \ar[l, shift right = 1ex, "\varGamma" swap]
 \end{tikzcd}
\]
where $N(-)$ is the normalization functor. These are compatible with natural model structures on the source and target. In particular, $N$ induces an equivalence on homotopy categories:
\[
N\colon \mathrm{Ho}(R) \xrightarrow{\ \sim\ } \cat{D}_{\geqslant 0}(R)\,.
\]
Here $\cat{D}_{\geqslant 0}(R)$ is the derived category on non-negatively graded chain complexes, which can be identified with the full-subcategory of $\dcat R$ consisting of $R$-complexes with homology concentrated in non-negative degrees. It is not a triangulated subcategory, for it is not closed under negative suspensions.

Thanks to the equivalence above, one can work with simplicial $R$-modules, at least for objects in $\cat{D}_{\geqslant 0}(R)$.

For a general simplicial ring $\mathcal A$, its normalization $N(\mathcal A)$ is a graded-commutative differential graded algebra, concentrated in non-negative degrees, and the Dold-Kan functor induces an equivalence between $\mathrm{Ho}({\mathcal A})$ and $\cat{D}_{\geqslant 0}(N(\mathcal A))$, the derived category of differential graded $N(\mathcal A)$-modules with homology concentrated in non-negative degrees. This subsumes the case $\mathcal A=R$ discussed above.
\end{chunk}

\begin{chunk}
\label{ch:formality}
Let $R$ be a ring. An $R$-complex $M$ is said to be \emph{formal} if there is an isomorphism $M\simeq \hh(M)$ in $\dcat R$. Given the Dold-Kan equivalence, we say a simplicial $R$-module is formal if its normalization is formal. For instance, when $k$ is a field, any simplicial $k$-vector space is formal.
\end{chunk}

\begin{chunk} 
\label{ch:postnikov}
Given a simplicial ring $\mathcal A$ and an integer $n\ge 0$, there is a simplicial ring $\mathcal B$ and a map of simplicial rings $\varphi \colon \mathcal A\to \mathcal B$ with the following properties:
\begin{enumerate}[\quad\rm(1)]
 \item 
 $\hh_i(\mathcal B)=0$ for $i\ge n+1$;
 \item
 $\hh_i(\varphi)$ is bijective for $i\le n$.
\end{enumerate}
The map can be obtained by a process of killing the homology in $\mathcal A$ in degree $n+1$ and higher; see also the discussion on \cite[pp.~162]{Toen:2010} for the construction of $\varphi$. This is part of the data of a Postnikov tower for $\mathcal A$.
\end{chunk}

\begin{chunk}
\label{ch:change-of-rings}
 Given a map $\vp\colon \mathcal A\to \mathcal B$ of simplicial rings, one has an adjoint pair of functors
\[
\begin{tikzcd}
\mathrm{Ho}(\mathcal A) \ar[r, shift left=1ex, "\vp^*" ] 
 & \mathrm{Ho}(\mathcal B) \ar[l, shift left = 1ex, "\vp_*" ]
\end{tikzcd}
\]
where $\vp_*$ is restriction along $\vp$, and $\vp^*$ is induced by $\mathcal B\otimes_{\mathcal A}-$. When $\vp$ is an equivalence, $\vp^*$ and $\vp_*$ are inverse equivalences of categories. The normalization $N(\mathcal A)$ is a differential graded algebra, and the following diagram 
\[
\begin{tikzcd}
 \mathrm{Ho}(\mathcal A) \ar[d, "\simeq", "N" swap] 
 & \mathrm{Ho}(\mathcal B) \ar[l, "\vp_*" ] \ar[d, "N", "\simeq" swap] \\
 \cat{D}_{\geqslant 0}(N(\mathcal A)) & \cat{D}_{\geqslant 0}(N(\mathcal B)) \ar[l, "\vp_*" ]
\end{tikzcd}
\]
commutes up to isomorphism of functors. This observation is used in recasting the statement in \cref{th:nilpotence} in terms of simplicial algebras and modules.
\end{chunk}

\begin{proof}[Proof of \cref{th:nilpotence}]
One has that $\fm\cdot \hh(K^M)=0$, where $\fm$ is the maximal ideal of $R$. Thus, with $\widehat{R}$ the $\fm$-adic completion of $R$, the map 
\[
K^M \longrightarrow \widehat{R}\otimes_R K^M \cong K^{\widehat R\otimes M}
\]
induced by the natural map $R\to \widehat{R}$, is a quasi-isomorphism. Thus passing to $\widehat R$, we can assume $R$ is complete in the $\fm$-adic topology. Let $Q\to R$ be a minimal Cohen presentation. 

We view $R$ as a simplicial $Q$-algebra in the usual way, and take a simplicial $R$-algebra model for $K^R$. Let $\rho\colon Q\{X\}\xra{\simeq}R$ and $Q\{Y\}\xra{\simeq}k$ be simplicial free resolutions of $R$ and $k$, respectively, as $Q$-algebras. Then 
\[
Q\{X,Y\}\colonequals Q\{X\}\otimes_QQ\{Y\}
\]
is a simplicial free resolution of $K^R$ over $Q\{X\}$. Let $F^e\colon Q\{X\}\to Q\{X\}$ be the Frobenius map applied degreewise. At this point we have constructed the top two squares of the following commutative diagram of simplicial $Q$-algebras:
\[
\begin{tikzcd}
R\ar[r,"F ^e"] & R \ar[r] & K^R \\
Q\{X\}\ar[u,"\simeq" swap, "\rho"] \ar[r,"F ^e"] \ar[drrr, "\Phi" swap] \ar[d,"\varepsilon" swap] 
 & Q\{X\} \ar[u,"\simeq", "\rho" swap] \ar[r] 
 & Q\{X,Y\} \ar[u,"\simeq" swap] \ar[r,twoheadrightarrow,"\simeq"] 
 & k\{X\} \ar[d, "\simeq", "\alpha" swap] \\
k \ar[rrr,dotted, "\Psi" swap]& & & \mathcal{A}
\end{tikzcd}
\]
The equivalence $Q\{X,Y\}\xra{\simeq} k\{X\}$ is obtained by applying $Q\{X\}\otimes_Q-$ to the equivalence $Q\{Y\}\xra{\simeq} k$. Next we construct the lower part of the diagram.

Let $J$ be the kernel of the canonical augmentation $k\{X\}\to k$; this is a simplicial ideal in $k\{X\}$ and satisfies $\hh_i(J)=0$ for $i\le 0$. Since $k\{X\}$ is free, as a simplicial $k$-algebra, for each integer $n\ge 0$ one has $\hh_i(J^{n+1})=0$ for $i\le n$, by Quillen's theorem~\cite[Theorem~6.12]{Quillen:1970}. In particular the surjection $ k\{X\} \to k\{X\}/J^{n+1}$ is bijective in homology in degrees $\le n$. Set $c\colonequals \codepth R$;
thus
\[
\hh_i(k\{X\})\cong \hh_i(K^R)=0 \quad\text{for $i>c$.}
\]
Let $k\{X\}/J^{c+1}\to \mathcal{A}$ be a map of simplicial rings such that the induced map in homology is bijective in degrees $\le c$ and $\hh_i(\mathcal A)=0$ for $i\ge c+1$; see~\cref{ch:postnikov}. The map $\alpha$ in the diagram is the composition of maps
\[
k\{X\} \longrightarrow k\{X\}/J^{c+1} \longrightarrow \mathcal{A}
\]
By construction $\hh(\alpha)\colon \hh(k\{X\})\to \hh(\mathcal{A})$ is an isomorphism; that is to say, $\alpha$ is an equivalence. The map $\Phi\colon Q\{X\}\to\mathcal A$ is defined to ensure that the upper triangle commutes.
Let $\varepsilon \colon Q\{X\}\to k$ be the canonical augmentation and set $I\colonequals \Ker(\varepsilon)$. Since $p^e\ge c+1$, the composition of maps
\[
Q\{X\}\xrightarrow{\ F ^e\ } Q\{X\}\longrightarrow Q\{X,Y\}\longrightarrow k\{X\}
\]
takes $I$ into $J^{c+1}$, so $\Phi$ factors through $\varepsilon$ yielding the map $\Psi$ in the diagram. This completes the construction of the commutative diagram. Observe that $k\{X\}$, and hence also $\mathcal A$ is a simplicial $k$-algebra, and that the map $\Psi\colon k\to \mathcal A$ is the composition of $F ^e\colon k\to k$ with the structure map $k\to \mathcal A$.

At this point, we complete the proof under the additional assumption that the $R$-complex $M$ satisfies $\hh_i(M)=0$ for $i<0$, for we can then work entirely in the simplicial context, given the Dold-Kan equivalence between non-negatively graded $R$-complexes and simplicial $R$-modules; see \cref{ch:dold-kan} and \cref{ch:change-of-rings}. To tackle the case of a general $M$, one can apply the normalization functor and get a commutative diagram of differential graded algebras and argue as below, but in the context of differential graded algebras and modules.

Since $k$ is a field, any simplicial $k$-module, is formal, in the sense of \cref{ch:formality}. Thus, for any simplicial $\mathcal A$-module $\mathcal M$, the simplicial $R$-module $U\colonequals \rho^*\varepsilon_*\Psi_*(\mathcal M)$ is formal. It thus follows from the commutative diagram above that for any simplicial $K^R$-module $\mathcal N$, the simplicial $R$-module $F ^e_*(\mathcal N)$ is formal. This applies in particular to the simplicial model for $K^M$, and so the argument is complete.

As to the last part, if $\hh(K^M)\ne 0$, then $\hh(F ^e_*(K^M)) \cong F ^e_* \hh(K^M)$, has $F ^e_*k$, and hence $k$, as a direct summand.
\end{proof}

The essence of the argument above is the existence of the commutative diagram of simplicial algebras, and this can be paraphrased as follows.

\begin{proposition}
\label{pr:nilpotence}
Let $R$ be a noetherian local ring of prime characteristic $p$. For any natural number $e> \log_p(\codepth R)$, in the homotopy category of simplicial commutative algebras, the composition 
\[
R\xra{F ^e}R\to K^R
\]
factors through the natural surjection $R\to k$ to the residue field. \qed
\end{proposition}

Given \cref{th:nilpotence}, and more so the result above, it follows that any property of the ring that can be characterized in terms of some homological property of $k$ can often be characterized in terms of the corresponding homological property of the Frobenius pushforwards, $F ^e_*M$, for any $M$ with the maximal ideal in its support. Kunz's characterization of regularity~\cite[Theorem~2.1]{Kunz:1969} of $R$ is the prototypical such result, where $M=R$. See also \cref{co:kunz}.
While some of these consequences can be deduced using results already available in the literature, in particular those  in \cite{Avramov/Iyengar/Miller:2006} or \cite{Koh/Lee:1998}, the approach outlined above lays bare the underlying phenomenon.

\section{Schemes}
\label{se:generation}

In this section we build on the results in \cref{se:nilpotence} to establish statements concerning generation of bounded derived categories; see \cref{th:strong-generation}. The main new input is the local-to-global principle from \cref{se:local-global}. To state our result we extend the definition of codepth for a local ring from ~\cref{se:nilpotence} to any noetherian scheme $X$ as follows:
\[
\codepth X \colonequals\sup\{\codepth {\mco}_{x} \mid x\in X\}\,.
\]
When $X$ is quasi-compact, excellent and of finite Krull dimension, there is bound on the embedding dimensions of the local rings $\mco_x$ and so also on $\codepth X$; this follows from \cite[Proposition~5.2]{Schoutens:2013}. However we get no concrete bounds from this source. We would like to have such a bound, at least for $F$-finite schemes; for this see ~\cref{le:codepth} below. To begin with, we record the following observation that means that for a local ring the codepth in the sense above is the same as the one introduced earlier.

\begin{lemma}
For any noetherian local ring $R$, one has $\codepth R \ge \codepth R_\fp$ for any prime ideal $\fp$ in $R$. 
\end{lemma}

\begin{proof}
The desired statement is clear when $R$ has a Cohen presentation $Q\to R$, which one can arrange to be minimal, for then \eqref{eq:codepth-pdim} yields (in)equalities
\[
\codepth R = \pdim_QR \ge \pdim_{Q_\fp}(R_\fp) \ge \codepth R_\fp\,,
\]
which is the desired inequality. 

We reduce to this case by passing to the completion, $\widehat R$, of $R$ at its maximal ideal. Pick a prime ideal, say $\fp'$ in $\spec\widehat R$ lying over $\fp$ and minimal with that property. One then has
\[
\edim (\widehat R_{\fp'}) \ge \edim R_\fp \quad\text{and}\quad \depth (\widehat R_{\fp'}) = \depth R_\fp\,,
\]
where the inequality holds by, for example, \cite{Lech:1964} and the equality by \cite[Proposition~1.2.16]{Bruns/Herzog:1998}. This gives the last of the following inequalities
\[
\codepth R = \codepth \widehat R \ge \codepth (\widehat R_{\fp'}) \ge \codepth R_{\fp}\,.
\]
The equality is clear, whereas the first inequality is by the discussion in the previous paragraph, for $\widehat R$ has a Cohen presentation.
\end{proof}

In the statement below, for any finitely generated $R$-module $M$ we write $\beta^R(M)$ for the minimal number of elements required to generate $M$. When $R$ contains a field of positive characteristic, we say $R$ is \emph{$F$-finite} if the Frobenius endomorphism $F\colon R\to R$ is a finite map. We use often the observation that quotients, localizations, and completions of $F$-finite noetherian rings are $F$-finite, and that the Frobenius endomorphism commutes with localization and completion.

\begin{lemma}
\label{le:codepth}
Let $R$ be a noetherian ring.
\begin{enumerate}[\quad\rm(1)]
\item When $R$ is a quotient of a regular ring $Q$, one has
\[
\codepth R\le \pdim_QR<\infty\,.
\]
\item
When $R$ contains a field of prime characteristic and is $F $-finite one has
\[
\edim R_\fp \le \beta^R(F _*R)\quad\text{for each $\fp\in\spec R$.}
\]
In particular, $\codepth R<\infty$.
\end{enumerate}
\end{lemma}

\begin{proof}
(1) When $R$ is a quotient of a regular ring $Q$ one has inequalities
\[
\codepth R_\fp \le \pdim_{Q_{\fp\cap Q}}R_\fp\le \pdim_Q R<\infty\,,
\]
for each prime ideal $\fp$ of $R$, where the first one is by \eqref{eq:codepth-pdim}, and the finiteness of $\pdim_QR$ is by \cite{Bass/Murthy:1967}. Thus $\codepth R$ is finite, and bounded above by $\pdim_Q R$. 

(2) We can deduce this from (1) for any $F $-finite ring is a quotient of an $F$-finite regular ring, as was proved by Gabber~\cite[Remark 13.6]{Gabber:2004}. Here is a direct argument, based on the proof of \cite[Proposition 1.1]{Kunz:1976}.

Since the number of generators does not go up under localization, the desired result is that when $(R,\fm,k)$ is a noetherian local ring $\beta^R(F _*R)\ge \edim R$. This follows from the computation:
\begin{align*}
\beta^R(F _*R) 
 & = \rank_k (k\otimes_R F _*R) \\
 & = \rank_k (F _*(R/\fm^{[p]}))\\
 &\ge \rank_k(F _*k)\rank_k(R/\fm^{[p]}) \\
 & \ge \rank_k(\fm/\fm^2)\,. 
\end{align*}
The (in)equalities are all standard.
\end{proof}

The preceding result gives one family of schemes whose codepth is finite. This family includes schemes that may have infinite Krull dimension.

\begin{corollary}
\label{co:finite_codpeth}
If $X$ is a scheme essentially of finite type over a regular scheme, then $\codepth X$ is finite. \qed
\end{corollary}

For the present, the more pertinent result is the one below.

\begin{proposition}
\label{pr:codepth-scheme}
If $X$ is a noetherian $F$-finite scheme, then $\codepth X$ is finite. 
\end{proposition}

\begin{proof}
Since $X$ admits a finite open affine cover, and $\codepth X$ is computed locally, we can assume $X$ is affine. The desired result follows from \cref{le:codepth}(2).
\end{proof}
 
These considerations are pertinent to the result below; it contains \cref{ith:general-schemes}. See \cref{se:local-global} for notation and terminology.

\begin{theorem}
\label{th:strong-generation}
Let $X$ be an $F$-finite noetherian scheme over a field of prime characteristic $p$. Fix a natural number $e> \log_p(\codepth X)$. If $E,G$ belong to $\dbcat{\coh X}$ with $\supp_XE\subseteq \supp_XG$, then 
\[
E\text{ is in }\thick^{\odot}_{\dcat X}(F_*^e{G})\,.
\]
Hence, if $\supp_X G=X$, then $F_*^eG$ generates $\dbcat{\coh X}$ as a $\perf X$-module.
\end{theorem}

\begin{proof}
By \cref{th:local-global} it suffices to check that $\res{E}x$ is in the thick subcategory of $\dcat{\mco_{x}}$ generated by $(F _*^eG)_x$ for $x\in X$. Fix a point $x\in X$.

If $x\notin \supp_XG$, then $\supp_XE\subseteq \supp_XG$ implies $E_x\cong 0$, hence also that $\res{E}x\cong 0$, so the desired inclusion is clear.

Suppose $x\in \supp_XG$. Set $R\colonequals \mco_{x}$, let $k$ be its residue field, and set $M\colonequals G_x$; thus $M\not\cong 0$ in $\dcat R$. By the definition of $\codepth X$ one has $p^e> \codepth R$ so \cref{th:nilpotence} applied to $M$ yields that $k$ is in the thick subcategory generated by $F ^e_*(K^M)$, and since $F ^e_*(K^M)$ is finitely built by $F^e_*(M)$, it  follows that $k$ is finitely build by $F^e_*(M)$. This implies that any $R$-complex with homology of finite length, and in particular $\res{E}x$, is in that thick subcategory as well. It remains to observe that $F ^e_*(M)\cong F ^e_*(G)_x$, as $R$-complexes. 
\end{proof}

Next we establish a result akin to \cref{th:strong-generation} that deals with the case when $X$ is not necessarily $F $-finite. It concerns the homotopy category of injectives that appeared already in \cref{ch:kinj-X}. We need also the functor
\[
\bfi \colon \dcat X \longrightarrow \kinj X
\]
that assigns to each complex its injective resolution; see \cite[Section 4.3]{Krause:2022}. 

\begin{theorem}
\label{th:kinj-generation}
Let $X$ be a noetherian scheme over a field of prime characteristic $p$. Fix a natural number $e> \log_p(\codepth X)$ and a complex $G$ in $\dcat{X}$ with $\hh(G)$ bounded and $\supp_XG=X$. As $\perf X$-modules one has
 \[
\Loc^{\odot}_{\cat K(X)}(\bfi F ^e_*G)=\kinj X\,.
 \]
\end{theorem}

This statement is vacuous when the codepth of $X$ is infinite. A construction of Hochster, see \cite[Proposition~1]{Hochster:1973}, produces affine noetherian schemes of infinite codepth. However, as noted before, codepth is finite for excellent schemes of finite Krull dimension. See also \cref{co:finite_codpeth}.

\begin{proof}
By the local-to-global principle \cref{eq:local-global} it suffices to verify that for each $x$ in $X$ the localizing $\perf X$-submodule of $\kinj X$ generated by $(\bfi F ^e_*G)_x$ contains the objects of $\kinj X$ supported at $x$. Moreover, since $\hh^i(G)=0$ for $|i|\gg 0$, one gets the first isomorphism below: 
\[
(\bfi F ^e_*G)_x \cong \bfi((F ^e_*G)_x)\cong \bfi F ^e_*(G_x)\,.
\]
The second one is standard. Thus passing to the local ring at $x$, we arrive in the situation where $R\colonequals \mco_x$ is a noetherian local ring, say with maximal ideal $\fm$, and $G$ is an $R$-complex with $\hh(G)$ bounded and $\supp_RG=\spec R$. The task is to verify that any $R$-complex in $\kinj R$ supported at $\fm$ is in the localizing subcategory of $\kinj R$ generated by $\bfi F ^e_*G$. This is equivalent to checking that if an object $M$ of $\kinj R$ is supported at $\fm$ and satisfies
\[
\Hom_{\cat K}^*(\bfi  F ^e_*G , M)=0
\]
then $M\cong 0$ in $\kinj R$; that is to say, all its syzygy modules are injective. 

Since $K^G$ is finitely built from $G$ in $\dcat R$, it follows that $F ^e_*(K^G)$ is in the thick subcategory of $\dcat R$ generated by $F ^e_*G$. Then, given the lower bound on $e$, \cref{th:nilpotence} yields that $k$, the residue field of $R$, and hence also any $R$-module of finite length, is in the thick subcategory generated by $F ^e_*G$. Since $\bfi(-)$ is an exact functor, we conclude that for each finite length $R$-module $L$, its injective resolution $\bfi L$ is in the thick subcategory of $\kinj R$ generated by $\bfi F ^e_*G$. Thus the condition above implies that
\[
\Hom_{\cat K}^*(\bfi L, M)=0\,,
\]
for all such $L$. The category of $\fm$-power torsion $R$-modules is a locally noetherian Grothendieck category with noetherian objects the finite length $R$-modules. Thus, since $M$ is $\fm$-power torsion, applying \cite[Lemma~2.2]{Krause:2005} we deduce from the condition above that $M\cong 0$ in $\kinj R$; see also \cite[Lemma~6.4.11]{Krause:2022}.
\end{proof}

\begin{chunk}
\label{ch:kinj-remark}
One can deduce \cref{th:strong-generation} from the result above: When $X$ is $F $-finite, if $G$ is in $\dbcat{\coh X}$ so is $F ^e_*G$ and hence $\bfi F ^e_*G$ is compact in $\kinj X$; see \cref{eq:kinjc}. The localizing submodule generated by $\bfi F ^e_*G$ equals $\kinj R$, so by \cref{eq:neeman-tt}
\[
\thick^{\odot}(\bfi F ^e_*G ) = {\kinj X}^{\sfc}\,.
\]
Applying the quotient functor $\boldsymbol{q}$ to the previous equality yields the desired result, given the equivalence~\cref{eq:kinjc}.
\end{chunk}

Next we discuss alternative formulations of the preceding results, in terms of strong generators of $\dbcat{\coh X}$.

\begin{lemma}
\label{le:thick-Frobenius}
Let $X$ be a noetherian scheme over a field of prime characteristic $p$. Let $U\subseteq \perf X$ be such that $\thick_{\dcat X}(U)=\perf X$. For any $G$ in $\dcat X$ and integer $e\ge 0$ one has
\[
\thick^{\odot}_{\dcat X}(F ^e_*G) 
 = \thick_{\dcat X}(F ^e_*(E\otimes G)\mid E\in U)\,.
\]
\end{lemma}

\begin{proof}
It suffices to verify this when $e=1$; this eases up the notation a bit. 

Let $F^*\colon \dcat{X}\to \dcat{X}$ denote pullback along $F$. For any collection of objects $C$ in $\dcat X$ with $\Loc_{\dcat X}(C)=\dcat{X}$, adjunction yields $\Loc_{\dcat X}(F^*C)=\dcat{X}$. From this observation and the fact that $F^*$ restricts to an endofunctor on $\perf{X}$, one can apply \cref{eq:neeman-cgt} to conclude that 
\begin{equation}\label{eq:pullback}
\thick_{\dcat X}(F^*U)=\thick_{\dcat X}(U)=\perf{X}\,.
\end{equation}
Now it remains to observe that 
\begin{align*}
 \thick_{\dcat X}(F _*(E\otimes G)\mid E\in U)&=\thick_{\dcat X}(F _*(F^*(E)\otimes G)\mid E\in U)\\
 &=\thick_{\dcat X}(E\otimes F _*(G)\mid E\in U)\\
 &=\thick^{\odot}_{\dcat X}(F_*G)\,;
\end{align*}
where the first equality is by \cref{eq:pullback}, the second equality is the projection formula
\[
F _*(F ^*(-)\otimes G)\simeq (-)\otimes F _*G
\]
on $\perf{X}$, and the last equality is evident since $\thick_{\dcat X}(U)=\perf{X}$. 
\end{proof}

\subsection*{Strong generation}
Here is our main result concerning the existence of strong generators for the bounded derived category of $X$. By a result of Bondal and Van den Bergh, see \cref{ch:BondalVanDenBergh}, a perfect complex $E$ as in the hypothesis always exists; when $X$ is affine, one can take $G=\mco_X$.

\begin{corollary}
\label{co:strong-generation}
Let $X$ be an $F$-finite noetherian scheme of prime characteristic $p$, and assume $E$ generates $\perf X$. For each $e>\log_p(\codepth X)$ and $G$ in $\dbcat{\coh X}$ with $\supp_XG=X$, the complex $F^e_*(E\otimes G)$ is a generator for $\dbcat{\coh X}$; it is a strong generator when $X$ is separated.
\end{corollary}

\begin{proof}
Since $\supp_X G= X$, \cref{th:strong-generation} yields the first equality below:
\[
\dbcat{\coh X}=\thick^{\odot}_{\dcat X}(F^e_*G)=\thick_{\dcat X}(F^e_*(E\otimes G))\,.
\]
The second equality is by \cref{le:thick-Frobenius}, given that $E$ generates $\perf X$. 

Since $X$ is $F$-finite, it has finite Krull dimension, by~\cite[Proposition~1.1]{Kunz:1976}---see also \cref{le:codepth}(2)---and is excellent, by~\cite[Theorem~2.5]{Kunz:1976}. When $X$ is separated, Aoki's theorem~\cite{Aoki:2021} yields that $\dbcat{\coh X}$ has a strong generator, and as $F^e_*(E\otimes G)$ is generator for $\dbcat{\coh X}$, it must be a strong generator as well.
\end{proof}

Here is a corollary of the preceding result.

\begin{corollary}
\label{co:quasi-proj}
Let $X$ be a quasi-projective scheme over an $F $-finite field of prime characteristic $p$, $L$ a very ample line bundle, and set
\[
G\colonequals \bigoplus_{i=0}^{\dim X}{L}^{\otimes i}\,.
\]
Then $F^e_*G$ strongly generates $\dbcat{\coh X}$ for each $e> \log_p(\codepth X)$.
\end{corollary}

\begin{proof}
Since $G$ generates $\perf X$, see \cite[Theorem~4]{Orlov:2009}, \cref{co:strong-generation} applies.
\end{proof}

The result below is immediate from \cref{co:strong-generation}, as $\perf X$ is generated by the tensor unit $\mco_X$ in the affine case.

\begin{corollary}
 \label{co:strong-generation-affine}
If $X$ is an $F$-finite affine scheme of prime characteristic $p$, then $F^e_*\mco_X$ strongly generates $\dbcat{\coh X}$ for each $e> \log_p(\codepth X)$. In fact $F^e_*G$ is a strong generator for $\dbcat{\coh X}$ whenever $G$ is in $\dbcat{\coh X}$ with $\supp_X G=X.$ \qed
\end{corollary}

\begin{chunk}
It follows from the preceding result that for any $F $-finite normal toric ring $R$ the isomorphism classes of conic modules generate $\dbcat{\rmod R}$. This is because each $R$-module $F ^e_*R$ decomposes as a direct sum of conic modules, see \cite[Proposition~4.15]{Faber/Muller/Smith:2019} (or \cite[section 3]{BrunsConic}), and there are only finitely many conic modules up to isomorphism; cf.\@ \cite[Corollary 4.12]{Faber/Muller/Smith:2019}, as well as \cite[Proposition 3.6]{BrunsGubeladze}.
\end{chunk}

One consequence of the existence of generators is the following strengthening of Kunz's theorem~\cite[Theorem~2.1]{Kunz:1969} characterizing regularity in terms of flatness of the Frobenius. The result below is, in turn, subsumed by \cite[Theorem~1.1]{Avramov/Hochster/Iyengar/Yao:2012} but it seems worthwhile to state and prove the version below for it highlights one application of statements concerning the existence of generators. 

\begin{corollary}
\label{co:kunz}
Let $X$ be an $F$-finite scheme. If there exists an $G$ in $\dbcat{\coh X}$ with $\supp_XG=X$ such that $F^n_*G$ is perfect for some $n\ge 1$, then $X$ is regular.
\end{corollary} 

\begin{proof}
A key observation observation is that given finite maps of schemes 
\[
X''\xrightarrow{f'}X'\xrightarrow{f} X
\]
and complexes $G''$ and $G'$ over $X''$ and $X'$, respectively, with $f'_*G''$ and $f_*G'$ are perfect, the complex $(ff')_*({f'}^{*}G'\lotimes G'')$ is perfect; the maps being finite guarantee that ${f'}^{*}G'\lotimes G''$ has coherent cohomology; see, for instance, \cite[Corollary~3.4]{Berthelot/Gorthendieck/Illusie:1971}. 

Given this observation it follows that if $F^n_*G$ is perfect for some $n\ge 1$, then $F^{2n}_*(F^{n,*}G\lotimes G)$ is perfect. Since $F$ is a homeomorphism on the underlying topological spaces, the support of $F^{n,*}G\lotimes G$ equals that of $G$. Thus, repeating this construction, we can make $n$ arbitrarily large; in particular, bigger than $\log_p(\codepth X)$. Let $E$ be a generator for $\perf X$. Using once again the observation in the previous paragraph, one deduces that $F^n_*(E\otimes G)$ is also perfect. Since  the complex generates $\dbcat{\coh X}$, by \cref{co:strong-generation}, it remains to apply \cref{le:sample} to deduce that $X$ is regular.
\end{proof}

\section{\texorpdfstring{$F$}{F}-thickness} 
\label{se:F-thickness}

The results in \cref{se:generation} guarantee that for any $F$-finite (separated) noetherian scheme $X$, its bounded derived category is (strongly) generated by a high Frobenius pushforward of \emph{some} perfect complex $G$. In this section we investigate when taking $G=\mco_X$ does the job. We refer the reader to \cite{Huybrechts:2006} regarding background on exceptional collections, semiorthogonal decompositions, and tilting generators.

\begin{definition}
\label{def:F_thick}
Let $X$ be a noetherian $F$-finite scheme. We say $X$ is \emph{$F$-thick} if $F_*^e \mco_X$ generates $\dbcat{\coh{X}}$, for some integer $e\ge 1$.
\end{definition}

By \cref{co:strong-generation-affine}, every affine scheme is $F$-thick. In this section, we present other examples of $F$-thick schemes, and examples of schemes that are not $F$-thick.

\subsection*{Relation to tilting}
The more stringent requirement that $F_*^e \mco_X$ is a tilting generator has been studied in \cite{Langer:2008, Raedschelders/Spenko/VandenBergh:2019,tilting_frob, frob_flag_1}. The class of $F$-thick schemes is strictly larger. Indeed for toric varieties there are obstructions to the existence of full exceptional collections, which prevents the tilting condition from being satisfied; see \cite[Theorem 1.3]{Efimov:2014}, as well as \cite{PerlingHille:2006,Mateusz:2011}. The affine situation is simpler.
 
\begin{proposition}\label{pr:regularity_via_tilting}
Let $R$ be a commutative noetherian ring. If $G$ is a tilting generator for $\dbcat{\rmod R}$, then $G\cong \shift^s P$, for some finitely generated, faithful, projective $R$-module $P$ and integer $s$; moreover, the ring $R$ is regular. Therefore $\dbcat{\rmod R}$ has a tilting generator if and only if $R$ is regular.
\end{proposition}

\begin{proof}
\cref{le:sample} implies $R$ is regular, and hence also that $\pdim_RG<\infty$. Then Nakayama's Lemma yields that
\[
\Ext_R^g(G,G) \ne 0 \text{ when } g=\pdim_RG - \inf \hh_*(G)\,.
\]
This implies the desired result.
\end{proof}

We record some examples of $F$-thick schemes.

\begin{example}
\label{exm:F_thick_projective_space}
 Let $X\colonequals \mathbb{P}^n_k$ denote projective $n$-space over $k$. In this case $X$ is $F$-thick and, moreover, $F_*^e \mco_X$ is a tilting generator whenever $p^e>n.$ 
 
 Indeed, \cite[Lemma~2.1]{Rao:1999} yields, for each $l\in \mathbb{Z}$, a decomposition
 \[
F_* ( \mco_X(l)) \cong \bigoplus_{i\geq -l/p}\mco_X (-i)^{\oplus \alpha(i,l)}\,, 
 \]
 where $\alpha(i,l)$ is the number of monomials of degree $l+ip$ that are not divisible by any $p^{\text{th}}$-power of a variable. Also, by \cite{beil_collection}, 
 \[
 G=\mco_X\oplus\mco_X (-1)\oplus\ldots\oplus \mco_X(-n)
 \]
 generates $\dbcat{\coh X}$. So combining these two facts yields that $F_*^e \mco_X$ generates $\dbcat{\coh X}$ whenever $p^e > n$. Also, as $\Ext_X^t(G,G)=0$ for $t>0$, it follows that $F_*^e \mco_X$ is a tilting generator for $\dbcat{\coh X}$.
 
 Fix $p^e \leq n$. For each $l$ in $\mathbb{Z}$ one has
 \begin{align*}
 \Hom_{\dcat{X}} (\mco_X(1), \shift^lF_*^e\mco_X)
 &\cong \Hom_{\dcat{X}} (\mco_X(p^e), \shift^l\mco_X) \\
 &= \hh^l (X, \mco_X(-p^e)) \\
 & = 0\,.
 \end{align*}
 Hence $F_*^e \mco_X$ does not generate $\dbcat{\coh X}$. 
\end{example}
	
\begin{example}
\label{exm:F_thick_alg_group}
 Let $\mathbf{G}$ be a semisimple algebraic group over an algebraically closed field of prime characteristic $p$. If $p>h$ where $h$ is the Coxeter number of $\mathbf{G}$, then $F_* \mco_\mathbf{G}$ is a
 generator for $\dbcat{\coh{\mathbf{G}}}$. Furthermore, if
 $\mathbf{P}$ is a parabolic subgroup of $\mathbf{G}$ and $p>h$, then
 $F_* \mco_{\mathbf{G}/\mathbf{P}}$ is a generator
 for $\dbcat{\coh{\mathbf{G}/\mathbf{P}}}$ (see \cite{Bezrukavnikov/Mirkovic/Rumynin:2008} and \cite[Corollary 1.1]{tilting_frob}). If $X$ is a flag variety of type $\mathbf{A}_2$ (or
 $\mathbf{B}_2$) and $p>3$ (or $p>5$, respectively), then $F_* \mco_X$ is a
 tilting generator for $\dbcat{\coh{X}}$; see \cite[Corollary 3.2]{tilting_frob}.
\end{example}
	
\begin{example}
\label{exm:F_thick_hara_blowup}
 The blowup $X$ of $\mathbb{P}^2$ at four points in general position is $F$-thick. By \cite[Corollary~5.4 \& Proposition~6.4]{Hara:2015}, if $e \geq  \log_p 3$, then $F_*^e \mco_X$ contains as direct summands bundles that form a full strong exceptional collection on $X$.
\end{example}
	
\begin{example}
Let $\mathbb{G}=G(2,n)$ be the Grassmannian over an algebraically closed field of prime characteristic $p$. By \cite[Corollary~16.11]{Raedschelders/Spenko/VandenBergh:2019} $F_* \mco_\mathbb{G}$ contains a full exceptional collection as direct summands whenever $p\geq n$, and hence it is a generator for $\dbcat{\coh \mathbb{G}}$; in fact, it contains the full exceptional collection in \cite{kaneda2,kaneda1}. By \cite[Theorem~1.1]{Raedschelders/Spenko/VandenBergh:2019} $F_* \mco_\mathbb{G}$ is a tilting generator if and only if $n=4$ and $p>3$.
\end{example}

\begin{example}\label{ex:severi_brauer_f_thick}
If $X$ is a Severi-Brauer variety over a perfect field $k$ of prime characteristic, then it is $F$-thick. We can apply \cite[Corollary~5.1.5]{Gille/Szamuely:2017} to obtain a finite Galois extension $\ell/k$ such that $X_{\ell}\cong \mathbb{P}^n_\ell$ for some integer $n>0$. The projection morphism $\pi\colon X_{\ell} \to X$ ensures that the exact functor
\[
\pi_* \colon \dbcat{\coh{X_{\ell}}} \to \dbcat{\coh{X}}
\]
is essentially surjective (see \cite{Sosna:2013} and \cite{Ballard/Duncan/McFaddin:2019}). There exists $e \gg 0$ such that $F_*^e \mco_{X_{\ell}}$ generates $\dbcat{\coh{X}_\ell}$. Therefore, $\pi_* (F_*^e \mco_{X_\ell})$ generates $\dbcat{\coh X}$. Notice that
\[
\pi_* F_*^e \mco_{X_\ell} \cong F_*^e \pi_* \mco_{X_\ell} \cong F_*^e \pi_* \pi^* \mco_X \cong (F_*^e \mco_X)^{\oplus g}\,,
\]
where $g\colonequals [\ell:k]$.
\end{example}

\begin{example}\label{ex:quadrics_f_thick}
If $X$ is a smooth quadric $X\subset \mathbb{P}^{n+1}$ over an algebraically closed field of odd characteristic $p$, then it follows from \cite[Corollary~5.2]{frob_flag_1} and \cite[Theorem~2]{Achinger:2012} $X$ is $F$-thick when $n$ is even and $n\geq 2(p+1)$, or $n$ is odd and $n \geq 3p+2$. The case where $F_*^e \mco_X$ is a tilting generator has been considered in \cite[Theorem~0.1]{Langer:2008}.
\end{example}

\begin{example}\label{ex:toric varieties}
In 2006, Bondal~\cite{Bondal:Fthick} suggested that toric varieties are $F$-thick. Recently, the first author speculated the stronger statement that $F_\ast^e \mco_X$ generates $\dbcat{\coh X}$ in exactly $\dim X$ steps for a toric variety $X$, provided that $e$ is sufficiently large. 
The latter was recently confirmed by Favero and Huang~\cite{FH:toric}, and Hanlon, Hicks, and Jeffs~\cite{HHL:toric}. 
\end{example}

\subsection*{Nonsingular curves}
Here we focus on nonsingular projective curves, and prove that the only ones which are $F$-thick are those of genus zero, and in this case, $F_*^e \mco_X$ is a tilting generator; see \cref{th:f_thick_curves}. This is in sharp contrast with the affine case; see \cref{co:strong-generation-affine}. First, a bit of terminology. 

Fix a nonsingular projective curve $X$, and let $E$ denote a vector bundle on $X$ of rank $r$. The slope of $E$ is $\mu(E ) \colonequals \frac{\deg(E )}{r}$ where $\deg(E)$ is its degree. We say $E$ is semistable (respectively, stable) if for all subbundles of positive rank $G \subseteq E $ defined over $k$, one has $\mu(G)\leq \mu(E)$ (respectively, $\mu(G) < \mu( E)$). 

\begin{theorem}
\label{th:f_thick_curves} 
A nonsingular projective curve over an infinite field of prime characteristic is $F$-thick if and only if its genus is zero.
\end{theorem}

\begin{proof}
The genus zero case is essentially \cref{exm:F_thick_projective_space}. 
	
Next, suppose that $X$ is an $F$-thick projective curve, and take $e>0$ so that $F_*^e\mco_X$ generates $\dbcat{\coh X}$. 
If $X$ has nonzero genus, then $F_*^e \mco_X$ is semistable by \cite[Corollary~4.4]{Sun:2008}, and so \cite[Theorem~1]{semistable_curves} ensures there exist a nontrivial vector bundle $E$ such that
\[
\hh^i(X, E\lotimes_{X} F_*^e \mco_X) =0\quad\text{for all $i$.}
\]
Denote by $E^\vee$ the dual vector bundle of $E$. From Serre duality, one gets
\[
\hh^i(X, E \lotimes_{X} F_*^e \mco_X) = \Ext^i_X(E ^\vee,F_*^e \mco_X)=0\,,
\]
which is a contradiction for $E^\vee$ is nontrivial and $F_*^e \mco_X$ generates $\dbcat{\coh X}$. Therefore, the genus of $X$ must be zero. 
\end{proof}
 
\begin{chunk}
A direct sum decomposition of $F_*^e \mco_X$ is known when $X$ is an elliptic curve (that is, the genus of $X$ is one). Namely, if $X$ is ordinary, then $F_*^e \mco_X$ splits into $p^e$ non-isomorphic $p^e$-torsion line bundles; cf.\@ \cite[Theorem~5.5]{Ejiri/Sannai:2016}. On the other hand, when $X$ is supersingular, $F_*^e \mco_X$ is isomorphic to Atiyah's indecomposable vector bundle $\mathcal{F}_{p^e}$ of degree zero and rank $p^e$; see \cite{Atiyah:1957}. 

For smooth projective curves $X$ of genus $g\geq 2$, it follows from \cite[Proposition~1.2]{LangePauly:2008} that the vector bundle $F_*^e \mco_X$ is stable and hence indecomposable.
\end{chunk}
 
\subsection*{Nonexamples and obstructions} 
Here is family of nonsingular projective varieties that are not $F$-thick.

\begin{theorem}
\label{th:ruled_surface_f_thick}
If $\pi \colon X \to C$ is a nonsingular projective ruled surface with $C$ a nonsingular projective curve with positive genus over an algebraically closed field of prime characteristic, then $X$ is not $F$-thick.
\end{theorem}
	
\begin{proof}
Assume the contrary that $F_*^e \mco_X$ is a generator for $\dbcat{\coh{X}}$ for some $e>0$. We claim that $C$ is $F$-thick, contradicting \cref{th:f_thick_curves}. 
	
To this end, the base field being algebraically closed allows us to use \cite[Theorem~2.6]{Orlov:1993} to obtain the following semiorthogonal decomposition 
\[
\dbcat{\coh{X}} = \langle \pi^*\dbcat{\coh{C}}, \mathcal{B}
\rangle\,.
\]
The corresponding projection 
\begin{equation}\label{eq:projection}
\dbcat{\coh{X}} \to \dbcat{\coh{C}}
\end{equation}
induces the isomorphism $\pi_* F_*^e \mco_X \cong F_*^e
\mco_C$ since $\pi_* \mco_X \cong \mco_C$ as coherent $\mco_C$-modules, and all higher derived pushforwards of $\pi_* \mco_X$ vanish. Finally, as \cref{eq:projection} is essentially surjective it follows that $F_*^e \mco_C$ is a generator for $\dbcat{\coh{C}}$, yielding the desired contradiction.
\end{proof}

The next result identifies obstructions to the $F$-thick property.
	
\begin{proposition}\label{pr:vanishing_f_pullback}
Let $X$ be a smooth projective variety over an $F$-finite field.  If there exists a nonzero object $E$ in  $\dbcat{\coh{X}}$ such that $F^{e,*}E$ is acyclic for some positive integer $e$, then $F_*^e \mco_X$ cannot generate $\dbcat{\coh{X}}$.
\end{proposition}

\begin{proof}
Standard isomorphisms using hom-tensor adjunction, the projection formula and Serre duality yield that
\[
 \Hom_{\dcat{X}} (\mco_X, \shift^n F^{e,*} E)\cong \Ext^{d-n}_X ( F_*^e \mco_X,\SheafHom(E,\omega_X) )^\vee
\]
vanish for each $n\in \mathbb{Z}$, where $d$ is the dimension of $X$. From these isomorphisms and the assumption on $F^{e,*}E$, we conclude that $F_*^e \mco_X$ does not generate  $\dbcat{\coh X}$.
\end{proof}

Let $X$ be a smooth projective $F$-thick variety. From \cref{pr:vanishing_f_pullback}, it follows that for every $E$ in  $\dbcat{\coh X}$ and $e \gg 0$, the complex $F^{e,*} E$ is not acyclic. In particular, we have the following constraint on the cohomology of line bundles.

\begin{corollary} \label{cor:neq-cohom-of-f-thick}
If $X$ is $F$-thick, then there is a nonnegative integer $n$ such that $\hh^*(X,L^n) \neq 0$ for any line bundle $L$.
\end{corollary}

\begin{proof}
Since $F^{e,*} L = L^{p^e}$, we can apply Proposition~\ref{pr:vanishing_f_pullback} along with the displayed isomorphism in the proof of \cref{pr:vanishing_f_pullback}.
\end{proof}

\begin{corollary}
Abelian varieties over algebraically closed fields of prime characteristic are not $F$-thick.
\end{corollary}

\begin{proof}
Let $X$ be an abelian variety and pick  $L \in \operatorname{Pic}^0(X)$ that is not $p$-torsion. As $\mco_X \neq L^{p^e}$, we have $\hh^*(X,F^{e,*}L) = 0$, so Corollary~\ref{cor:neq-cohom-of-f-thick} applies.
\end{proof}

\subsection*{Projective complete intersection varieties}
\label{sec:F_summands_CI}

All varieties considered in the rest of this section are assumed to be defined over a fixed $F$-finite field $k$ of characteristic $p$. For a projective complete intersection variety $X\subseteq \mathbb{P}_k^n$ of degree $d$, as is well-known, there a semiorthogonal decomposition
\begin{equation}
 \label{eq:KusnetsovSOD}
 \dbcat{\coh X} = \langle \mathcal A_X, \mco_X (d-n) ,\ldots,\mco_X \rangle\,;
\end{equation}
 $\mathcal A_X$ is often referred to as the Kuznetsov component of \cref{eq:KusnetsovSOD} due to the connection Kuznetsov~\cite{Kuznsetov:CubicFourFold} drew between the structure of $\mathcal A_X$ and the rationality of $X$. The component was also studied by Orlov~\cite[Theorem~2.3]{orlov_derivd_cats_sing} where he related it to the category of graded singularities associated to $X$. A scheme $X$ is said to be \emph{globally $F$-split} if $\mco_X$ is a summand of $F^e_*\mco_X$.

\begin{theorem}
\label{th:f_split_kuz_comp}
Let $X\subseteq \mathbb{P}^n_k$ be a projective complete intersection of degree at most $n$ and such $X$ is globally $F$-split. Then $X$ is $F$-thick if and only if $F_*^e \mco_X$ generates the Kuznetsov component of $X$ for some $e>0$.
\end{theorem}

Contrast this with \cref{th:completeintersection} where we prove $F_*G$ generates $\dbcat{\coh X}$ for some perfect complex $G$. \cref{th:f_split_kuz_comp} follows from the following stronger result, describing $ F_*^e \mco_X $ for such $X$. We thank Devlin Mallory for suggesting this formulation, and for supplying part of the proof.

\begin{theorem}
\label{th:f_split_sod}
Let $X$ be as in \cref{th:f_split_kuz_comp} and let $d$ be its degree. If $p^e> n-d$ then there is a direct sum decomposition 
\[
 F_*^e \mco_X = M\oplus \bigoplus_{i\in\mathbb{Z}}  \mco_X(i)^{a_i}\,,
\]
where $M$ has no twists of $\mco_X$ as direct summands, and 
\[
a_i= \begin{cases} 
1 & \text{if $i=0$} \\
\ge 1 & \text{if $d-n\le i\le -1$}\\
0 &\text{otherwise}.
\end{cases}
\]
Moreover, no other line bundles show up in the direct sum decomposition of $F_*^e\mco_X$ when $\dim X\ge 3$.
\end{theorem}

\begin{chunk}
 Any $X$ as in \cref{th:f_split_sod} is Fano. If $X$ is smooth it is well-known that it is $F$-split for $p \gg 0$. More generally for any rational singularity  $\tilde{X}$ with a model over a finitely-generated $\mathbb{Z}$-algebra,  its reduction mod $p$ is $F$-split for $p \gg 0$; see \cite{Smith:2000}.
\end{chunk}

We fix the following notation for the rest of the section. Let $S\colonequals k[x_0,\ldots,x_n]$ where each $x_i$ has degree one, and write $\fm$ for the homogeneous maximal ideal $(x_0,\ldots,x_n)$ of $S$. Fix $I$ a homogeneous ideal of $S$ and set $R\colonequals S/I$. 

\begin{chunk}
\label{lem:fedder_like_criterion}
Write $I^{[p^e]}$ for the ideal generated by all $r^{p^e}$ where $r\in I$. The element $F^e_*r$ of $F^e_*R$ refers to the elements $r \in R$ viewed in $F^e_*R$. We record a variation on Fedder's criterion: For a homogeneous element $s$ in $S$, with image $r$ in $R$, and  for any positive integer $e$, the inclusion of the cyclic $R$-submodule of $F_*^eR$ generated by $F_*^e (r)$ splits if and only if $s ( I^{[p^e]}:I) \not \subseteq \mathfrak{m}^{[p^e]}$.
The argument is similar that in \cite[Proposition~1.7]{Fedder:1983}; see also \cite{Glassbrenner:1996}.
\end{chunk}

\begin{proposition}
\label{prop:general_graded_fedder} 
As graded $R$-modules $R(-j)$ is a direct summand of $F_*^e R$ if and only if there exists $s \in S_{p^ej}$ such that
\[
 s (I^{[p^e]} : I) \not\subseteq \fm^{[p^e]}\,.
\]
\end{proposition}
	
\begin{proof}
We consider the $\frac{1}{p^e}\mathbb{Z}$-graded structure on $F^e_*R$: for a homogeneous $r \in R$, the corresponding element in $F^e_*R$ has degree $\frac{\text{deg}(r)}{p^e}$. The direct sum of all $\mathbb{Z}$-graded components of $F_*^e R$ is denoted $(F_\ast^e R)_\mathbb{Z}$. It can be shown that a $\mathbb{Z}$-graded $R$-module $M$ is a direct summand of $(F_\ast^e R)_\mathbb{Z}$ if and only if it is a direct summand of $F_\ast^e R$.
 
Suppose that $R(-j)$ is a direct summand of $(F_*^e R)_\mathbb{Z}$ as a $\mathbb{Z}$-graded $R$-module, and let $\phi\colon R(-j) \to (F_*^e
R)_\mathbb{Z}$ be a splitting. The element $r \in R$ where $F^e_* r = \phi(1)$  has degree $p^ej$ and from \cref{lem:fedder_like_criterion} any lift $s$ of $r$ satisfies $s (I^{[p^e]} : I) \not\subseteq \fm^{[p^e]}$.
	
Conversely, suppose that there exists an $s \in S_{p^ej}$ such that
$s (I^{[p^e]} : I) \not\subseteq \fm^{[p^e]}.$
Write $r$ for the image of $s$ in $R_{p^ej}$. The element $F_*^e (r)$ in $(F_*^e R)_\mathbb{Z}$ has degree $d$, and so the map of $R$-modules $R \to (F_*^e R)_\mathbb{Z}$ defined by $1 \mapsto F_*^e (r)$ induces a $\mathbb{Z}$-graded $R$-module map
$\phi\colon R(-j) \to (F_*^e R)_\mathbb{Z}$. The assumption on $s$, combined with \cref{lem:fedder_like_criterion} ensures that $\phi$ splits in the category of (ungraded) $R$-modules. Hence it also splits in the category of graded modules. 
\end{proof}

\begin{proof}[Proof of \cref{th:f_split_sod}]
Assume $q\colonequals p^e > n-d$. First we verify that  $\mco_X(a)$ is a summand of $F_*^e\mco_X$ for each $a$ such that $d-n\le a\le 0$; this part of the proof works also when $X$ has degree $n+1$. Note $\mco_X(a)$ is a summand of $F^e_*\mco_X$ if and only if $R(a)$ is an $R$-module summand of $F^e_*R$; see \cite[Theorem 3.10]{Karen2000Vanishing}. We produce necessary $R$-module splittings using \Cref{prop:general_graded_fedder}. Let  $f_1,\ldots, f_t$ be a regular sequence in $S$ that generates $I$, and  defines $X$ in $\mathbb{P}^n$. Then
\[
(I^{[q]}: I)= (f^{q-1})+ I^{[q]} \quad \text{where} \,\, f= f_1\cdots f_t \,. 
 \]
Since $X$, and hence also $R$, is $F$-split, $f^{q-1} \notin m^{[q]}$; see  \cite[section 2]{Fedder:1983}, \cite[Prop 4.10]{Karen1997SC}. Let $g$ be a nonzero monomial in $S$ such that
\[
 g f^{q-1} \notin m^{[q]}\quad \text{but}\quad x_igf^{q-1} \in m^{[q]}  \quad\text{for each } i=0,\ldots,n\,.
 \]
Thus $(x_0\cdots x_n)^{q-1}$ appears in the homogeneous polynomial $g f^{q-1}$, and so its degree is $(q-1)(n+1)$. For $q> n-d$, one has
\[
(n-d)q \leq \text{deg}(g)= (n+1)(q-1)- d(q-1) \, .
\]
Thus for  $j=0,\ldots, n-d$ we can choose a monomial factor $s_j$ of $g$ of degree $jq$ with $s_jf^{q-1} \notin m^{[q]}$. Then \Cref{prop:general_graded_fedder} yields that  $\mco_X(-j)$ is a summand of $F^e_*\mco_X$.

We verify, by an induction on $t$, that $\oplus_{j=n-d-t}^{n-d} \mco_X(-j)$ is a summand of $F^e_*\mco_X$ for $t \geq 0$; this holds when $t=0$ by the argument above.
 Once we know 
\[
F^e_*\mco_X \cong E \oplus \bigoplus_{j=n-d-t}^{n-d}\mco_X(-j) \,,
\]
it follows that $\mco_X(n-d-t-1)$ is a summand of $E$ since the former is a summand of $F^e_* \mco_X$ and does not admit any nonzero map to $\oplus_{j=n-d-t}^{n-d}\mco_X(-j)$. This completes the proof of the claim about the twists of $\mco_X$ occurring in $F^e_*\mco_X$.

We now show that $\mco_X(a)$ can appear in the direct sum decomposition of $F_*^e\mco_X$ only if $d-n\le a\le 0$ and exactly one copy of $\mco_X$ splits. To that end, write
\[
F_*^e \mco_X = M\oplus \bigoplus_{i\in\mathbb{Z}}  \mco_X(i)^{a_i} 
\]
where $M$ is some coherent sheaf. Since $X$ is $F$-split, $\mco_X$ is a summand of $F^e_* \mco_X$. Since $\hh^0(X,\mco_X)=\hh^0(X,F_*^e \mco_X)$ and the global sections on the right side are already accounted for by one $\mco_X$ summand, $a_0=1$ and $\mco_X(a)$ cannot appear as a direct summand of $F_*^e \mco_X$ for $a\ge 1$.  As to the other summands, since $\omega_X$ is invertible  one has the first isomorphism below:
\[
\SheafHom(F^e_*\omega_X, \mco_X) \otimes \omega_X \cong \SheafHom(F^e_*\omega_X, \omega_X) \cong F^e_*\SheafHom(\omega_X, F^{e, !}\omega_X)\cong F^e_*\mco_X \,.
\]
The others are standard. So $\mco_X(b)$ appears as a summand of $F^e_*\mco_X$ if and only if 
 \[
 \mco_X(-b) \otimes \omega_X \cong \mco_X(-b) \otimes \mco_X(d-n-1) = \mco_X(d-n-1-b)
 \]
 is a summand of $F^e_*\omega_X$. Since $\omega_X$ is anti-ample, $F^e_*\omega_X$ has no nonzero global sections, whereas $\mco_X(a)$ does whenever $a\ge 0$. We conclude that if $\mco_X(b)$ is a summand of $F^e_* \mco_X$ then $b\ge d-n$. 

 It remains to note that when $\dim X\ge 3$ the Picard group of $X$ is $\mathbb{Z}\cdot \mco_X (1)$. 
\end{proof}

\section{Numerical invariants}\label{se:new_invariants}

Motivated by the work in \cref{se:generation} one can introduce new invariants coming from the derived category to measure singularity types. One such invariant is briefly studied in this section. Throughout $R$ is an $F$-finite noetherian ring.

\begin{definition}
\label{d:F-level}
The \emph{Frobenius level} (or simply $F$-\emph{level}) of $R$ is
\[
\flevel(R)\colonequals \inf\{n\geq 0\mid R\text{ is in }\thick_{\dcat R}^n(F_*^eR) \text{ for some }e>0\}\,.
\]
\cref{co:strong-generation-affine} implies that this number is finite. A natural question arises: Does $F_*R$ already generate $R$? That is to say, is $F_*R$ proxy-small? See \cite{Dwyer/Greenlees/Iyengar:2006a,Dwyer/Greenlees/Iyengar:2006b} for a discussion on proxy-smallness, and some consequences of this property.
\end{definition}

\begin{chunk}
In the terminology of \cite{Avramov/Buchweitz/Iyengar/Miller:2010}, the \emph{level} of an $R$-complex $M$ with respect to an $R$-complex $N$, denoted $\level^N_{\dcat R}(M)$, is the infimum of the set of integers $n$ such that $M$ is in $\thick_{\dcat R}^n (N)$. Therefore $\flevel(R)$ is the least value of the level of $R$ with respect to $F_*^eR$ for a positive integer $e$.
\end{chunk}

A ring $R$ is \emph{$F$-split} if the Frobenius map $F\colon R\to R$ is split in $\rmod R$. The Frobenius level is a measure of the failure of $F$-splitness. This follows from the result below; it is well-known, we sketch a proof for lack of an adequate reference. 
	
\begin{proposition}
\label{pr:F_splitting_level}
For any finite map $f\colon R\to S$, with $R$  commutative noetherian, the following conditions are equivalent:
\begin{enumerate}[\quad\rm(1)]
    \item 
    $f\colon R\to S$ is split as $R$-modules;
    \item
    $f_*S$ has $R$ as a direct summand;
    \item
    $(f_*S)^n$ has $R$ as a direct summand for some $n\ge 1$;
    \item $\level^{f_*S}(R)\le 1$.
\end{enumerate}
\end{proposition}

\begin{proof}
Clearly (1)$\Rightarrow$(2)$\Rightarrow$(3)$\Rightarrow$(4). Moreover, (4) means that $R$ is a direct summand of an object of the form 
 \[
 \bigoplus_{n\in \mathbb{Z}} \shift^n (f_*S)^{\oplus s_n}\,;
 \]
equivalently,  $R$ is a direct summand of $(f_*S)^{\oplus s_0}$; thus (4)$\Rightarrow$(3). 

 It remains to verify $(3)\Rightarrow(1)$. As (1) can be checked locally on $\spec R$, so we can assume $R$ is local. Then (3) implies $(f_*S)^n$, and hence also $f_*S$ is a faithful $R$-module, and moreover that this remains so under base change. Thus $R\to S$ is pure; since it is finite, we deduce that it is split; see, for instance, \cite[Theorem~7.14]{Matsumura:1989}.
\end{proof}
 

\begin{example}
Let $R$ be an artinian $F$-finite local ring with residue field $k$, and let $\ell\ell(R)$ denote the Loewy length of $R$. For $e\ge \log_p\ell\ell(R)$ it is easy to see that $F_*^eR$ is a nonzero finite $k$-vectorspace and hence, 
\[
\level_{\dcat{R}}^{F_*^eR}(R)=\level_{\dcat{R}}^{k}(R)=\ell\ell(R)\,;
\]
see, for example, \cite[Theorem~6.2]{Avramov/Buchweitz/Iyengar/Miller:2010}. Thus $\flevel(R)\le \ell\ell(R)$.
\end{example}

Here is another family of rings for which we could bound $F$-levels. The bound is sharp, for the result below and \cref{{pr:F_splitting_level}}, any $F$-finite non-$F$-split hypersurface over a field of characteristic two has Frobenius level exactly 2. 

\begin{proposition}
\label{th:flevel_lci}
If $R$ is an $F$-finite locally complete intersection ring of prime characteristic $p$, then 
$\flevel(R)\le p^{\codepth R}$.
\end{proposition}

\begin{proof}
 By \cite[Corollary~3.4]{Letz:2021}, 
 \[
 \level^{F_*R}_{\dcat R}(R)=\sup\{\level^{F_*R_\fp}_{\dcat {R_\fp}}(R_\fp)\mid \fp\in \spec R\}
 \]
 so we can assume $R$ is a local complete intersection ring. Applying \cite[Corollary~2.12]{Letz:2021}, we can further assume $R$ is a quotient of an $F$-finite regular local ring $S$ by a regular sequence $f_1,\ldots,f_c$. It suffices to show $\flevel(R)$ is at most $p^c.$ 
 
Factor $F$ as $R\to R'\to R$
where $R'=S/(f_1^p,\ldots, f_c^p)$ and the first map is base change along the Frobenius of $S$. Since $S$ is $F$-finite and regular, $R'$  is a nonzero finite free $R$-module. By \cite[Remark~2]{Avramov/Miller:2001}, one has a filtration
\[
0=R'_{p^c}\subseteq \ldots \subseteq R_1'\subseteq R_0'=R'
\]
by $R'$-submodules where each subquotient is $F_*R$. In particular, 
\[
\level^{F_*R}_{\dcat R}(R')\le p^c\,,
\]
since $R'$ is free over $R$.
\end{proof}

\begin{chunk}
It seems interesting to consider the $F$-level of dualizing complexes.
 Clearly, there are other invariants that one can introduce to study the singularity type, also in the global context. In future work, we hope to explore the relations between them and properties like finite Frobenius representation type and strong $F$-regularity.
\end{chunk}

\section{The first Frobenius pushforward}
\label{se:firstpushforward}

The results from \cref{se:generation} also bring up the following question.

\begin{question}
\label{qu:first-push-forward}
Suppose $X$ is $F$-finite. Is there a perfect complex $G$ such that $F_*G$ itself generates $\dbcat {\coh X}$? 
\end{question}

We know this is so when $X$ is regular. The main result of this section, \cref{th:completeintersection}, generalizes this to locally complete intersections schemes.

 \cref{qu:first-push-forward} also brings to mind the result below, due to Mathew~\cite{Mathew:2018}. 

\begin{chunk}
Let $R$ be a noetherian ring of prime characteristic and $F $ its Frobenius endomorphism. Assume $F $ is finite. Since the kernel of $F $ consists of nilpotent elements, it follows from \cite[Theorem~3.16]{Mathew:2018}, see also \cite[Section~11]{Bhatt/Scholze:2017a}, that for any integer $e\ge 1$ the $R$-module $F ^e_*R$ generates $\dcat R$ as a tensor ideal thick subcategory:
\[
\dcat R = \thick_{\dcat R}^{\otimes}(F ^e_*R)\,.
\]
\end{chunk}

\subsection*{Locally complete intersections}
\label{se:lci}

A scheme $X$ is \emph{locally complete intersection} if it is noetherian and the local ring $\mco_x$ is complete intersection for each $x\in X$. This condition means that the completion of $\mco_x$ at its maximal ideal can be presented as a regular ring modulo a regular sequence; see \cite[Section~2.3]{Bruns/Herzog:1998} for details.

\begin{theorem}
\label{th:completeintersection}
When $X$ is $F$-finite and locally complete intersection, for any perfect complex $G$ that generates $\perf X$, the complex $F_*G$ is a generator for $\dbcat{\coh X}$. In particular, if $X$ is affine, $F_*\mco_X$ is a strong generator for $\dbcat{\coh X}$.
\end{theorem}

The proof uses the theory of cohomological support varieties for complete intersection local rings. The results we need are recalled below; see \cite{Avramov:1989} and \cite{Avramov/Buchweitz:2000b}. 

\begin{chunk}
\label{ch:supportvarieties}
Let $R$ be a local complete intersection ring of codimension $c$, and residue field $k$. For each $M$ in $ \dbcat{\rmod R}$, one can associate a Zariski closed cone $\rmV_R(M)$ in the homogeneous spectrum $\spec^* \mcS$ where $\mcS$ is the symmetric algebra on the graded $k$-space $\shift^{-2}k^c$; see \cite{Avramov:1989}. This is the ring of cohomology operators of Gulliksen~\cite{Gulliksen:1974} and Eisenbud~\cite{Eisenbud:1980}. The (Krull) dimension of $\rmV_R(M)$ is the complexity of $M$: 
\[
\cx_R(M)=\inf\{d\in \mathbb{N}: \rank_k \Ext_R^n(M,k)\le a n^{d-1}\text{ for some $a>0$ and all $n$}\}\,.
\]
The relevance of these varieties to the problem at hand is because of \cite[Theorem~3.1]{Liu/Pollitz:2021}: If $R$ is locally complete intersection and $M,N$ are in $ \dbcat{\rmod R}$, then 
\[
N\text{ is in }\thick_{\dcat R} M\iff \rmV_{R_\fp}(N_\fp)\subseteq \rmV_{R_\fp}(M_\fp) \text{ for each } \fp\in \spec R\,.
\]
\end{chunk}

\cref{th:completeintersection} follows from the result below with $G=\mco_X$, and \cref{le:thick-Frobenius}.

\begin{theorem}
\label{th:completeintersection-2}
Let $X$ be an $F $-finite noetherian scheme. If $X$ is locally complete intersection, each $G$ in $\dbcat{\coh X}$ with $\supp_X G= X$ satisfies
\[
\dbcat{\coh X}=\thick^{\odot}_{\dcat X}(F_*G)
\]
as a $\perf X$-module.
\end{theorem}

\begin{proof}
Fix $x\in X$, set $R\colonequals \mco_x$ and $M\colonequals G_x$. Given \cref{co:local-global}, it is suffices to verify the $R$-complex $F _*M$ generates $\dbcat{\rmod R}$.

Let $N$ be in $ \dbcat{\rmod R}$ and $\fp\in \spec R$. By \cite[Theorem~12.2.4]{Avramov/Iyengar/Miller:2006},
\[
\cx_{R_\fp}(F _*(M_\fp))=\cx_{R_\fp}(k(\fp)) =\mathrm{codim} R_\fp\,.
\]
In particular, $\rmV_{R_\fp}(N_\fp)\subseteq \rmV_{R_\fp}(F _*(M_\fp))$. Since $(F _*M)_\fp$ is isomorphic to $F _*(M_\fp)$ we obtain that 
\[
\rmV_{R_\fp}(N_\fp)\subseteq \rmV_{R_\fp}((F _*M)_\fp)\,.
\]
As the previous equality holds for each $\fp\in \spec R$ we can apply \cref{ch:supportvarieties} to conclude that $N$ is in $\thick_{\dcat R}(F _*M)$. 
\end{proof}

 The next result produces examples of rings $R$ which are not locally complete intersection, yet $F _*R$ generates $\dbcat{\rmod R}$.

\subsection*{Veronese subrings}
Let $k$ be an $F $-finite field of prime characteristic and let $S\colonequals k[x_1, \ldots, x_d]$, the polynomial ring in indeterminates $x_1,\dots,x_d$, each having degree one. The observation below is applied when $R$ is a Veronese subring of $S$.

\begin{lemma}
\label{le:S-generates}
Let $\iota\colon R\to S$ be a finite extension that is split as a map of $R$-modules.
Then $\dbcat{\rmod R}= \thick_{\dcat R}(\iota_*S)$. 
\end{lemma}

\begin{proof}
It suffices to show that any finitely generated $R$-module $M$ is in $\thick_{\dcat R}(S)$. Since $S$ is regular, $S \otimes_R M$ is in $\thick_{\dcat S}(S)$ and hence in $\thick_{\dcat R}(\iota_*S)$ by restriction of scalars along $\iota$. Since $\iota$ is split, $M$ is an $R$-module summand of $S \otimes_RM $. Thus $M$ is finitely built from $\iota _* S$ in $\dcat{R}$. 
\end{proof}

The minimal graded free resolution of the $S$-module $S/(x_1, \ldots, x_d)^j$ has the form:
\begin{equation}
\label{eq:linear-resolution}
\begin{tikzcd}
0 \rightarrow S^{\oplus b_d(j)}(-j-d+1)\rightarrow \cdots \rightarrow S^{\oplus b_1(j)}(-j) \rightarrow S^{\oplus b_0(j)} \rightarrow 0 \,.
\end{tikzcd}
\end{equation}
The resolution is linear after the first term. The integers $b_i(j)$ are given by
\[
b_i(j) =
\begin{cases}
1 & i=0\\
\frac{(j+d-1)!}{(j-1)!(d-i)!(i-1)!(j+i-1)} & i \leq d \\
0 & i\ge d+1\,.
\end{cases}
\]
See, for instance, \cite{Buchsbaum/Eisenbud:1975}. Fix an integer $\ell\ge 1$. For each integer $j$, let $G_j$ denote the $k$-vector space spanned by all monomials in $S$ whose total degree is congruent to $j$ modulo $\ell$. Note that each $G_j$ is a finitely generated $R$-module. 
Taking strands of \cref{eq:linear-resolution} yields exact sequences of $R$-modules 
\begin{equation}\label{eq:exact-sequence-strands}
 0 \rightarrow G_{\ell-d+1}^{\oplus b_d(j)} \rightarrow \cdots G_{l-1}^{\oplus b_2(j)}\rightarrow R^{\oplus b_1(j)} \rightarrow G_j \rightarrow 0
 \end{equation}
 for $1 \leq j \leq \ell -1$.

\begin{lemma}
\label{le:veronsese-generators}
Let $R$ be the $\ell$-th Veronese subring of $k[x,y]$. For each integer $1\le j\le \ell-1$ the following equality holds
 \[
 \dbcat{\rmod R}= \thick_{\dcat R} (R\oplus G_j)\,.
 \]
 \end{lemma}

\begin{proof}
Given \cref{le:S-generates} and the isomorphism $S \cong G_0\oplus \ldots \oplus G_{\ell-1}$
of $R$-modules, it suffices to verify that each $G_i$ is in $\thick_{\dcat R}(R\oplus G_j)$ for $0\le i\le \ell-1$. This becomes clear by applying \Cref{eq:exact-sequence-strands} with $d=2$.
\end{proof}

\begin{proposition}
\label{pr:veronese}
Let $k$ be an $F $-finite field. If $R$ is the $\ell$-th Veronese subring of $k[x,y]$, for some $\ell \ge 1$, then $F _*R$ is a strong generator for $\dbcat{\rmod R}$. 
\end{proposition}

\begin{proof}
 By \cref{le:veronsese-generators} it suffices to verify that $\thick_{\dcat R}(F _*R)$ contains $R$ and at least one $G_j$ where $1 \leq j \leq \ell-1$. The first containment is satisfied as the Frobenius map $R \rightarrow F _*R$ splits as an $R$-module, as $R$ is a direct summand of a regular ring.
 
 To show some $G_j$ is in $\thick_{\dcat R}(F _*R)$, we import results from the theory of conic modules over a toric ring; see, for example, \cite{BrunsGubeladze,Faber/Muller/Smith:2019}. Namely, the module $F _*R$ is isomorphic to a direct sum of conic modules; see \cite[Proposition~4.15]{Faber/Muller/Smith:2019} and \cite[Theorem~3.1]{BrunsConic}. On the other hand, any conic $R$-module is isomorphic to some $G_i$ by \cite[Remark~7.6]{Faber/Muller/Smith:2019}. Therefore since $G_0\cong R$ and $R$ is not regular, $F _*R$ must have a summand isomorphic to $G_j$ for some $1\leq j\leq \ell -1.$
\end{proof}

\begin{chunk}
Let $S= k[x,y]$ and $R$ be the $\ell$-th Veronese subring of $S$. For $\ell\geq 3$, the ring $R$ is not a complete intersection ring. Moreover, if $\ell>p^2[k:k^p]$, then $F _*R$ cannot have $S$ as a free $R$-module summand; indeed, the rank of $S$ and $F _*R$ as $R$-modules are $\ell$ and $p^2[k:k^p]$ respectively, cf.\@ \Cref{le:S-generates}.
\end{chunk}

We end with a few observations regarding \cref{qu:first-push-forward}.

\begin{chunk}
\label{ch:qu-observations}
By the local-to-global principle from \cite[Corollary~3.4]{Letz:2021}, to settle \cref{qu:first-push-forward} in the affine setting, it suffices to check that when $R$ is local $F_*R$ generates the residue field of $R$. When $R$ is also Cohen-Macaulay, one can reduce to $R$ artinian.

Suppose $R$ is an artinian local ring. \cref{qu:first-push-forward} has a positive answer when $k$ is a direct summand of $F _*R$. This property holds if and only if the maximal ideal $\fm$ of $R$ satisfies $(0:\fm^{[p]})\not\subseteq \fm^{[p]}$. For example, if $R=\mathbb{F}_2[x,y]/(x^4,x^2y^2,y^4)$ it is not immediately clear whether $F _*R$ generates $\dbcat{\rmod R}$ since $k$ is not a summand of $F_*R$; in this example, 
\[
F _*R\cong (R/(x^2,xy,y^2))^{\oplus 3}\,.
\]
However a direct analysis shows that even for this ring $F_*R$ generates $\dbcat{\rmod R}.$
\end{chunk}

\noindent \textbf{Conflicts of Interest:} None.

\noindent \textbf{Financial Support:} This material is based on work supported by Simons Foundation: Collaboration Grant for Mathematicians \#708132 (Ballard) and Grant \#586594 (Lank) and work supported by the National Science Foundation under Grants No. DMS-1840190 (Pollitz), DMS-FRG-1952399, DMS-2101075 (Mukho-padhyay), DMS-2001368 (Iyengar), DMS-2002173 (Pollitz), DMS-2302567 (Pollitz), DMS-2302263 (Ballard). This was also funded by ERC Advanced Grant 101095900-TriCatApp (Lank). Part of it was done at the Hausdorff Research Institute for Mathematics, Bonn, when Iyengar and Pollitz were at the ``Spectral Methods in Algebra, Geometry, and Topology" trimester, funded by the Deutsche Forschungsgemeinschaft under Germany's Excellence Strategy--EXC-2047/1--390685813. Additionally, Ballard, Iyengar, and Lank were supported by the National Science Foundation under Grant No. DMS-1928930 while the authors
were in residence at the Simons Laufer Mathematical Sciences Institute (formerly MSRI).

\bibliographystyle{amsplain}
\bibliography{refs}

\end{document}